\documentclass[12pt]{article}

\usepackage[english]{babel}
\usepackage[utf8x]{inputenc}
\usepackage[T1]{fontenc}

\usepackage{mathtools}

\usepackage[a4paper,top=3cm,bottom=2cm,left=3cm,right=3cm,marginparwidth=1.75cm]{geometry}

\usepackage[vcentermath]{youngtab}

\usepackage[shortlabels]{enumitem}

\usepackage[vcentermath]{youngtab}
\usepackage{amsthm}

\usepackage{amsmath}
\usepackage{amssymb}
\usepackage{graphicx}
\usepackage[colorinlistoftodos]{todonotes}
\usepackage[colorlinks=true, allcolors=blue]{hyperref}
\usepackage{pst-node}
\usepackage{tikz-cd} 
\usepackage{xfrac} 

\usepackage{graphicx}
\usepackage{cleveref}

\newcommand{\commentaires}[1]{}

\makeatletter
\newcommand*\bigcdot{\mathpalette\bigcdot@{.5}}
\newcommand*\bigcdot@[2]{\mathbin{\vcenter{\hbox{\scalebox{#2}{$\m@th#1\bullet$}}}}}
\makeatother

\newtheorem{definition}{Definition}[subsection]
\newtheorem{theorem}[definition]{Theorem}
\newtheorem{corollary}[definition]{Corollary}
\newtheorem{lemma}[definition]{Lemma}
\newtheorem{proposition}[definition]{Proposition}

\newtheorem{remark}[definition]{Remark}
\newtheorem{example}[definition]{Example}
\newtheorem{notations}[definition]{Notations}

\newtheorem{sect_definition}{Definition}[section]
\newtheorem{sect_theorem}[sect_definition]{Theorem}

\newtheorem{sect_lemma}[sect_definition]{Lemma}

\DeclareMathOperator{\grad}{grad}

\DeclareMathOperator{\SymF}{Sym}
\DeclareMathOperator{\Stab}{Stab}
\DeclareMathOperator{\Id}{Id}

\DeclareMathOperator{\tr}{tr}

\DeclareMathOperator{\Ree}{Re}

\DeclareMathOperator{\GL}{GL}

\DeclareMathOperator{\Mat}{Mat}

\DeclareMathOperator{\reg}{reg}
\DeclareMathOperator{\Rep}{Rep}

\DeclareMathOperator{\gl}{\mathfrak{gl}}

\title{Trivializations of moment maps}
\author{Mathieu Ballandras \\ \it{Université de Paris} \\ \it{Scuola Internazionale Superiore di Studi Avanzati} \\ \small{mballandras@imj-prg.fr} }

\begin{document}

\newcommand\gv{\mathfrak{g}_v}
\newcommand{\uv}{\mathfrak{u}_v}
\newcommand{\SU}{SU}
\newcommand\ZZ{\mathbb Z}
\newcommand{\Cbar}{\bold{\overline{C}}}
\newcommand{\Obar}{\bold{\overline{O}}}
\newcommand{\ql}{\overline{\mathbb{Q}}_l}
\newcommand{\C}{\mathbb{C}}
\newcommand{\RR}{\mathbb{R}}
\newcommand{\qlbar}{\overline{\mathbb{Q}}_l}
\newcommand{\cg}{\mathcal{C}(G)}
\newcommand{\cgd}{\mathcal{C}(\widehat{G})}
\newcommand{\F}{\mathcal{F}}
\newcommand{\Fi}{\mathbb{F}}
\newcommand{\sd}{\mathfrak{S}_d}
\newcommand{\sn}{\mathfrak{S}_n}
\newcommand{\Sym}{\mathfrak{S}}
\newcommand{\Fq}{\mathbb{F}_q}
\newcommand{\Fqbar}{\overline{\mathbb{F}}_q}
\newcommand{\QQ}{\mathbb Q }
\newcommand{\NN}{\mathbb N}
\newcommand{\LamX}{\Lambda \left[X\right]}
\newcommand{\SymX}{\SymF \left[X\right]}
\newcommand{\Part}{\mathcal{P}}
\newcommand{\Partbar}{\widebar{\mathcal{P}}}
\newcommand{\Ht}{\tilde{H}}
\newcommand{\SymnX}{\SymF_n\left[X\right]}
\newcommand{\inv}{^{-1}}
\newcommand{\NNp}{\NN_{>0}}
\newcommand{\typ}{\bold{T}}
\newcommand{\ptyp}{\bold{\tilde{T}}}
\newcommand{\dcb}[1]{\mathcal{D}_c^b\left(#1\right)}
\newcommand{\hi}{\mathcal H ^i}
\newcommand{\uic}[1]{\underline{\mathcal{I C}}_{#1}^\bullet}
\newcommand{\ic}[1]{\mathcal{I C}_{#1}^\bullet}
\newcommand{\obar}{\boldsymbol{\mathcal{\overline{O}}}}
\newcommand{\oo}{\boldsymbol{\mathcal{O}}}

\newcommand{\map}[5]{
\begin{array}{ccccc}
    #1 & : & #2 & \to & #3  \\
     & & #4 & \mapsto & #5
\end{array}
}

\newcommand{\gammabar}{\overline{\gamma}}
\maketitle
\begin{abstract}
We study various trivializations of moment maps. First in the general framework of a reductive group $G$ acting on a smooth affine variety. We prove that the moment map is a locally trivial fibration over a regular locus of the center of the Lie algebra of $H$ a maximal compact subgroup of $G$. The construction relies on Kempf-Ness theory \cite{kempf_ness} and Morse theory of the square norm of the moment map studied by Kirwan \cite{Kirwan}, Ness-Mumford \cite{Ness} and Sjamaar \cite{Sjamaar}. Then we apply it together with ideas from Nakajima \cite{nakajima1994} and Kronheimer \cite{kronheimer1989} to trivialize the hyperkähler moment map for Nakajima quiver varieties. Notice this trivialization result about quiver varieties was known and used by experts such as Nakajima and Maffei but we could not locate a proof in the literature.

\end{abstract}

\newpage

\tableofcontents

\newpage

\section{Introduction}
\subsection{Symplectic quotients and GIT quotients of affine varieties}
Consider a reductive group $G$ acting on a complex smooth affine variety $X$. For  $\chi^{\theta}\in\mathcal{X}^*(G)$  a linear character, $X^{\theta\text{-ss}}$ is the $\theta$-semistable locus and $X^{\theta\text{-s}}$ the $\theta$-stable locus. Mumford's geometric invariant theory \cite{mumford_git} provides a quotient
\[
X^{\theta\text{-ss}}\to X^{\theta\text{-ss}}// G.
\]
The affine variety $X$ can be embedded in an hermitian vector space $W$ such that the $G$-action is linear and restricts to a unitary action of a maximal compact subgroup $H\subset G$. The hermitian norm on $W$ is denoted by  $||\dots ||$. We study the associated real moment map
\[\mu : X \to \mathfrak{h}\] 
with $\mathfrak{h}$ the Lie algebra of $H$. Its definition relies on the choice of a non degenerate scalar product $\left\langle\dots,\dots\right\rangle$ on $\mathfrak{h}$ invariant under the adjoint action of $H$. The real moment map satisfies for all $Y\in\mathfrak{h}$
\begin{equation}\label{intro_eq_moment}
    \left\langle \mu(x),Y\right\rangle=\frac{1}{2}\left.\frac{d}{dt} ||\exp (i t Y).x||^2\right|_{t=0}
\end{equation}

Thanks to the invariant scalar product, to a linear character $\chi^{\theta}$ is associated an element $\theta$ in $Z(\mathfrak{h})$, the center of the Lie algebra $\mathfrak{h}$, such that for all $Y\in\mathfrak{h}$
\[
\left\langle \theta ,  Y \right\rangle = i d\chi^\theta_{\Id}(Y).
\]
For a pair $(\chi^{\theta},\theta)$, Kempf-Ness theory \cite{kempf_ness} relates the symplectic quotient (defined by Meyer \cite{Meyer} and Marsden-Weinstein \cite{Marsden_Weinstein}) to the GIT quotient, it gives an homeomorphism
\[
\mu\inv(\theta)/H \xrightarrow{\sim}  X^{\theta\text{-ss}}// G.
\]

We study trivialization of the moment map over a regular locus in the center of the Lie algebra $\mathfrak{h}$. First, in Section \ref{sect_kempf}, we study the general framework of a unitary action of a compact group on a smooth affine variety. After a reminder of Migliorini's version of Kempf-Ness theory \cite{migliorini}, a regular locus in $Z(\mathfrak{h})$ is defined. Over this locus the moment map is proved to be a locally trivial fibration. The case of a torus action was treated by Kac-Peterson \cite{Kac_Peterson}. The construction of the regular locus uses the negative gradient flow of square norm of the moment map studied by Kirwan \cite{Kirwan}, Ness-Mumford \cite{Ness}, Sjamaar \cite{Sjamaar}, Harada-Wilkin \cite{Harada} and Hoskins \cite{Hoskins}.

Nakajima quiver varieties introduced in \cite{nakajima1994} are particular instances of the symplectic quotients studied in Section \ref{sect_kempf}. Moreover they are hyperkähler quotients as defined by Hitchin-Karlhede-Lindström-Roček \cite{hklr}, the construction of those varieties is recalled in Section \ref{section_quiver}. In Section \ref{sect_hyperkahler_quiver}, the idea of Kronheimer \cite{kronheimer1989} and Nakajima \cite{nakajima1994} of consecutive use of different complex structures are applied together with techniques from previous sections to prove that the hyperkähler moment map is a locally trivial fibration. This implies in particular that the cohomology of the fibers forms a local system. This later result is used by Nakajima in \cite[Section 9]{nakajima1994} to construct a Weyl group action on the cohomology of quiver varieties. Maffei pursued this construction in \cite{Maffei}. I was informed by Nakajima that the property of the cohomology of the fibers can also be obtained by generalizing Slodowy argument from \cite{Slodowy_four} to quiver varieties. Similar results concerning cohomology of the fibers also exist in the framework of deformations of symplectic quotient singularities in Ginzburg-Kaledin \cite{ginzburg_kaledin}. Finally Crawley-Boevey and Van den Bergh \cite{CB_VDB} trivialize the hyperkähler moment map for Nakajima quiver varieties over complex lines. Nakajima explained to us how to extend their result to quaternionic lines minus a point thanks to the theory of twistor spaces see Theorem \ref{th_cbvdb}.

In the remaining of the introduction the results are stated and the various steps of the constructions are outlined.

\subsection{Real moment map for the action of a reductive group on an affine variety}
In Section \ref{sect_kempf}, $H\subset G$ is a maximal compact subgroup acting unitarily on a smooth affine variety $X$ embedded in an hermitian vector space. The differential geometry point of view from Kempf-Ness theory allows to extend the definition of $\theta$-stability for elements $\chi^{\theta}\in \mathcal{X}^*(G)^{\RR}:=\mathcal{X}^*(G)\otimes_{\ZZ} \RR$. The correspondence between linear characters and elements in the center of the Lie algebra $\mathfrak{h}$ thus extends to an isomorphism of $\RR$-vector spaces between $\mathcal{X}^*(G)^{\RR}$ and $Z(\mathfrak{h})$.

In \ref{subsect_HM} we prove a Lie group variant of Hilbert-Mumford criterion for $\theta$-stability. It is adapted to the differential geometric point of view of Kempf-Ness theory and the use of real parameters $\theta\in\mathcal{X}^*(G)^{\RR}$. Similar criteria are discussed by Georgoulas, Robbin and Salamon in \cite{GVR_Hilbert}.
\begin{theorem}[Hilbert-Mumford criterion for stability]
Let $\theta\in\mathcal{X}^*(G)^{\RR}$ and $x\in X$. The following statements are equivalent
\begin{enumerate}[label=(\roman*)]
\item\label{item:1}$x$ is $\theta$-stable.
\item  For all $Y\in\mathfrak{h}$, different from zero, such that $\lim_{t\to+\infty}\exp(i t Y).x$ exists  then $\left\langle\theta,Y\right\rangle < 0$.
\end{enumerate}
\end{theorem}
\noindent This theorem is applied in  \ref{subsect_King} to generalize a result of King \cite{king} characterizing $\theta$-stability for quiver representations.

The regular locus $B^{\reg}$ is introduced in \ref{subsect_regular_locus}. Its construction relies on the study of the negative gradient flow of the square norm of the moment map from Kirwan \cite{Kirwan}, Ness-Mumford \cite{Ness}, Sjamaar \cite{Sjamaar}, Harada-Wilkin \cite{Harada} and Hoskins \cite{Hoskins}. $B^{\reg}$ is an open subset of $Z(\mathfrak{h})$ such that for $\theta\in B^{\reg}$, one has $X^{\theta\text{-ss}}=X^{\theta\text{-s}}\ne\emptyset$ and for all $x\in X^{\theta\text{-s}}$ the stabilizer of $x$ is trivial. Over the regular locus, the moment map is a locally trivial fibration. A similar fibration when $G$ is a torus follows from a result of Kac-Peterson \cite{Kac_Peterson}. Let us also mention that with the flow of the norm square in the hermitian space $W$, Sjamaar \cite{Sjamaar} constructed a retraction of the $0$-stable locus to the fiber over $0$ of the moment map.
\begin{theorem}\label{th_intro_moment}
Let $\theta_0$ in $B^{\reg}$, and $U_{\theta_0}$ the connected component of $B^{\reg}$ containing $\theta_0$. There is a diffeomorphism $f$ such that the following diagram commutes

\begin{equation*}
\begin{tikzcd}
 U_{\theta_0} \times {\mu}\inv \left(\theta_0\right)\arrow[swap]{rd}{} \arrow[r,"f"',"\sim"] &  \mu\inv( U_{\theta_0} ) \arrow{d}{{\mu}_{}} \\
 &   U_{\theta_0} 
\end{tikzcd}
\end{equation*}
Moreover $f$ is $H$ equivariant so that the diagram goes down to quotient
\begin{equation*}
\begin{tikzcd}
 U_{\theta_0} \times {\mu}\inv \left(\theta_0\right)/H\arrow[swap]{rd}{} \arrow[r,"\sim"] &  \mu\inv( U_{\theta_0} )/H \arrow{d} \\
 &   U_{\theta_0} 
\end{tikzcd}
\end{equation*}

\end{theorem}
To prove this theorem, first we prove that for any $\theta\in U_{\theta_0}$ and $x\in X^{\theta_0\text{-s}}$ there exists a unique $Y(\theta,x)\in \mathfrak{h}$ such that $\exp(i Y(\theta,x)).x \in \mu\inv(\theta)$. This is achieved thanks to Migliorini's version of Kempf-Ness theory \cite{migliorini} which applies to affine varieties and real parameters $\chi^{\theta}\in \mathcal{X}^*(G)^{\RR}$.  Then the map $f$ is defined by 
\[
f(\theta,x):=\exp\left(i Y(\theta,x)\right).x
\] and similarly for its inverse 
\[
f\inv(x)=\left(\mu(x), \exp\left(i Y\left(\theta_0, x \right) \right).x\right).
\]
The smoothness of $f$ and its inverse is proved in \ref{subsect_triv_moment} with the implicit function theorem.

\subsection{Nakajima quiver varieties and hyperkähler moment map}
The quiver varieties considered in this paper were introduced by Nakajima \cite{nakajima1994}. Let $\widetilde{\Gamma}$ be an extended quiver with vertices $\Omega_0$ and edges $\widetilde{\Omega}$, fix a dimension vector $v\in \mathbb{N}^{\Omega_0}$. The space of representations of $\widetilde{\Gamma}$ with dimension vector $v$ is
\[
\Rep\left(\widetilde{\Gamma},v\right)=\bigoplus_{\gamma\in\widetilde{\Omega}}\Mat_{\C}(v_{h(\gamma)},v_{t(\gamma)}).
\]
with $h(\gamma)\in \Omega_0$ the head of the edge $\gamma$ and $t(\gamma)\in\Omega_0$ its tail. This space is acted upon by the group 
\[
G_v\cong \left\lbrace \left.(g_j)_{j\in {\Omega_0}}\in \prod_{j\in\Omega_0 }\GL_{v_j}\right| \prod_{j\in {\Omega_0}} \det(g_{j})=1 \right\rbrace.
\]
This action is described in  \ref{subsect_generalities}, it restricts to a unitary action of the maximal compact subgroup
\[
U_v=\left\lbrace \left.(g_j)_{j\in {\Omega_0}}\in \prod_{j\in {\Omega_0}} U_{v_j}\right| \prod_{j\in {\Omega_0}} \det(g_{j})=1 \right\rbrace
\]
with $U_{v_j}$ the group of unitary matrices of size $v_j$. Denote by $\uv$ the Lie algebra of $U_v$. This is a particular instance of the general situation of Section \ref{sect_kempf}: a unitary action of a compact group on a smooth complex affine variety. Let $\theta\in \mathbb{Z}^{\Omega_0}$ such that $\sum_j v_j \theta_j = 0$. Define $\chi^{\theta}$ a linear character of $G_v$ by
\begin{equation}
\chi^{\theta}\left((g_j)_{j\in\Omega_0}\right) := \prod_{j\in\Omega_0} \det(g_j)^{-\theta_j}.
\end{equation}
For quiver representations, the correspondence between linear characters and elements in the center of $\uv$ is easily described: to the character $\chi^{\theta}$ is associated the element $(-i\theta_j \Id_{v_j})_{j\in\Omega_0}\in Z(\uv)$. This element is still denoted by $\theta$, and $Z(\uv)$ is identified in this way with a subspace of $\RR^{\Omega_0}$.

A well-known theorem from King \cite{king} gives a characterization of $\theta$-stability for quiver representations. In \ref{subsect_King} this result is generalized to real parameters corresponding to elements $\chi^{\theta}\in\mathcal{X}^*(G)^{\RR}$.
\begin{theorem}
For $\theta\in\mathbb{R}^{\Omega_0}$ such that $\sum_{j\in\Omega_0} \theta_j v_j = 0$ and associated element $\chi^{\theta}\in\mathcal{X}^*(G_v)^{\RR}$. A quiver respresentation $(V,\phi)$ is $\theta$-stable if and only if for all subrepresentation $W\subset V$
\[
\sum_{j\in \Omega_0} \theta_{j} \dim W_j < 0.
\]
unless $W=V$ or $W=0$.
\end{theorem}

The space $\Rep\left(\widetilde{\Gamma},v\right)$ admits three complex structures denoted by  $I, J$ and $K$, they are detailed in \ref{subsect_hk_rep}. There is a real moment map for each one of this complex structure, they are denoted by  $\mu_I, \mu_J$ and $\mu_K$. They are defined as in equation \eqref{intro_eq_moment}, for instance
\[
    \left\langle \mu_I (x),Y\right\rangle=\frac{1}{2}\left.\frac{d}{dt} ||\exp (t.I. Y).x||^2\right|_{t=0}
\]
and
\[
    \left\langle \mu_J (x),Y\right\rangle=\frac{1}{2}\left.\frac{d}{dt} ||\exp (t.J. Y).x||^2\right|_{t=0}.
\]
Together they form the hyperkähler moment map $\mu_{\mathbb{H}}=(\mu_I,\mu_J,\mu_K)$, it takes values in $\uv^{\oplus 3}$.

Nakajima quiver varieties are constructed for $(\theta_I,\theta_J,\theta_K)\in Z(\uv)^{\oplus 3}$ as quotients of fibers of the hyperkähler moment map.
\[
\mathfrak{m}_v(\theta_I,\theta_J,\theta_K) = \mu_{\mathbb{H}}\inv(\theta_I,\theta_J,\theta_K) / U_v.
\]
The hyperkähler regular locus in $Z(\uv)^{\oplus 3}$ is defined by:
\begin{definition}[Hyperkähler regular locus]
For $w\in\mathbb{N}^{\Omega_0}$ a  dimension vector
\begin{equation*}
    H_w:=\left\lbrace(\theta_I,\theta_J,\theta_K)\in\left(\RR^{\Omega_0}\right)^3\left|\sum_{j} w_j \theta_{I,j} = \sum_j w_j \theta_{J,j}= \sum_j w_j \theta_{K,j} =0\right.\right\rbrace.
\end{equation*}
The regular locus is
\begin{equation}
    H_v^{\reg}=H_v \setminus \bigcup_{w< v} H_w
\end{equation}
the union is over dimension vector $w\ne v$ such that $0\le w_i\le v_i$.
\end{definition}
In \ref{subsect_triviality_hk} various trivializations of the hyperkähler moment map are discussed. We prove that the hyperkähler moment map is a locally trivial fibration by consecutive use of constructions of Theorem \ref{th_intro_moment} for each complex structure and associated moment map. The idea of consecutive use of different complex structures comes from Kronheimer \cite{kronheimer1989} and Nakajima \cite{nakajima1994}.
\begin{theorem}[Local triviality of the hyperkähler moment map]
Over the regular locus $H_v^{\reg}$, the hyperkähler moment map $\mu_{\mathbb{H}}$ is a locally trivial fibration compatible with the $U_v$-action:

Any $(\theta_I,\theta_J,\theta_K )\in H_v^{\reg} $ admits an open neighborhood $V$, and a diffeomorphism $f$ such that the following diagram commutes
\begin{equation*}
\begin{tikzcd}
V\times {\mu}_{\mathbb{H}}\inv (\theta_I,\theta_J,\theta_K )\arrow[swap]{rd}{} \arrow[r,"f"',"\sim"] &  \mu_{\mathbb{H}}\inv(V) \arrow{d}{\mu_{\mathbb{H}}} \\
 &  V
\end{tikzcd}
\end{equation*}
Moreover $f$ is compatible with the $U_v$-action so that the diagram goes down to quotient
\begin{equation*}
\begin{tikzcd}
V\times {\mu}_{\mathbb{H}}\inv (\theta_I,\theta_J,\theta_K )/U_v\arrow[swap]{rd}{} \arrow[r,"\sim"] &   \mu_{\mathbb{H}}\inv(V) /U_v \arrow[d,"p"] \\
 &  V
\end{tikzcd}
\end{equation*}

\end{theorem}

A similar trivialization of the hyperkähler moment map over lines is described in \cite[Lemma 2.3.3]{CB_VDB}. In Theorem \ref{th_cbvdb} we provide an extension of their result using twistor spaces as suggested by Nakajima.

Denote by $\pi$ the map obtained by taking quotient of the hyperkähler moment map over the regular locus
\[
\mu_{\mathbb{H}} \inv(H_v^{\reg})/U_v\xrightarrow{\pi} H_v^{\reg}.
\]
Consider $\mathcal{H}^i \pi_* \qlbar$, the cohomology sheaves of the derived pushforward of the constant sheaf. As a direct corollary of the local triviality of the hyperkähler moment map, those sheaves are locally constant. Moreover as $H_v^{\reg}$ is simply connected, those sheaves are constant. They provide the local system of the cohomology of the fibers.

\subsection*{Aknowledgement} I am very grateful to Hiraku Nakajima for his very useful comments and suggestions and for clarifying many historical facts. I want to thank Michèle Vergne for interesting discussions about the paper. The motivation of this work originated in discussions with Andrea Maffei about his paper \cite{Maffei}. I want to thank Florent Schaffhauser and Andrea Maffei for their very detailed and helpful proofreading. This work is part of my PhD thesis at Université de Paris and Scuola Internazionale Superiore di Studi Avanzati. I am thankful to my advisors Emmanuel Letellier and Fernando Rodriguez-Villegas for introducing me to quiver varieties and for their support.

\section{Kempf-Ness theory for affine varieties}\label{sect_kempf}
Kempf-Ness \cite{kempf_ness} relate geometric invariant theory quotients to symplectic quotients. In this section we recall Migliorini's version of this theory \cite{migliorini} which applies to affine varieties and real parameter $\chi^{\theta}\in\mathcal{X}^*(G)^{\mathbb{R}}$. Then we prove that the real moment map is a locally trivial fibration over a regular locus.

$G$ is a connected reductive group acting on a smooth affine variety $X$. The action is assumed to have a trivial kernel.
\subsection{Characterization of semistability from a differential geometry point of view}
For $\chi^{\theta}\in \mathcal{X}^*(G) $ a linear character of $G$, a regular function $f\in\C\left[X\right]$ is $\theta$-equivariant if there exists a strictly positive integer $r$ such that $f(g.x)=\chi^\theta(g)^r f(x)$ for all $x\in X$.
\begin{definition}\label{def_stab}
A point $x\in X$ is $\theta$-semistable if there exists a $\theta$-equivariant regular function $f$ such that $f(x)\ne 0$. The set of $\theta$-semistable points is denoted by  $X^{\theta\text{-ss}}$.

A point $x\in X$ is $\theta$-stable if it is $\theta$-semistable and if its orbit $G.x$ is closed in $X^{\theta\text{-ss}}$ and its stabilizer is finite. The set of $\theta$-stable points is denoted by  $X^{\theta\text{-s}}$.
\end{definition}
The GIT quotient as defined by Mumford \cite{mumford_git} is denoted by  $X^{\theta\text{-ss}}\to X^{\theta\text{-ss}}//G$. A point of this quotient represents a closed $G$-orbit in $X^{\theta\text{-ss}}$. When working over the field of complex numbers, such quotients are related to symplectic quotients. The affine variety $X$ can be embedded as a closed subvariety of an hermitian space $W$ with hermitian pairing denoted by  $p(\dots,\dots)$. The embedding can be chosen so that the action of $G$ on $X$ comes from a linear action on $W$ and the action of a maximal compact subgroup $H\subset G$ preserves the hermitian pairing, $p(h.u,h.v)=p(u,v)$ for all $h\in H$ and $u,v\in W$. Then $G$ can be identified with a subgroup of $\GL(W)$. The hermitian pairing induces a symplectic form on the underlying real space 
\begin{equation}\label{symp_herm}
\omega(\dots,\dots):=\Ree p(i\dots,\dots)
\end{equation}
with $i$ a square root of $-1$ and $\Ree$ the real part. The hermitian pairing on the ambient space induces an hermitian metric on $X$. As $X$ is a smooth subvariety of $W$, its tangent space is stable under multiplication by $i$, hence the non-degeneracy of the hermitian metric implies the non degeneracy of the restriction of the symplectic form $\omega$ to the tangent space of $X$ and the symplectic form on $W$ restricts to a symplectic form on $X$. Then the action of $G$ on $X$ induces a symplectic action of $H$ on $X$.

For $x\in X$ introduce the Kempf-Ness map
\begin{equation*}
    \map{ \phi^{\theta,x} }{ G }{ \RR }{ g }{ ||g.x||^2-\log\left(|\chi^\theta(g)|^2\right) }
\end{equation*}
with $||\dots||$ the hermitian norm.
\begin{theorem}[\cite{migliorini} Theorem A.4 ]\label{th_critical_stable}
A point $x_0\in X$ is $\theta$-semistable if and only if there exists in the closure of its orbit a point $x\in \overline{G.x_0}$ such that $\phi^{\theta,x}$ has a minimum at the identity.
\end{theorem}
\begin{remark}
Let $\mathcal{X}^*(G)^\RR :=\mathcal{X}^*(G)\otimes_{\ZZ}\RR$, the definiton of $\phi^{\theta,x}$ makes sense not only for linear characters but for any $\chi^\theta\in \mathcal{X}^*(G)^\RR$. It provides the following generalization of the definition of $\theta$-semistability and $\theta$-stability for any $\chi^\theta\in\mathcal{X}^*(G)^\RR$. 
\end{remark}
\begin{definition}[Semistable points]\label{def_real_stab}
Let $\chi^\theta\in\mathcal{X}^*(G)^\RR$, a point $x_0$ is $\theta$-semistable if there exists $x\in\overline{G.x_0}$ such that $\phi^{\theta,x}$ has a minimum at the identity.

A point $x_0$ is $\theta$-stable if it is $\theta$-semistable, its orbit is closed in $X^{\theta\text{-ss}}$ and its stabilizer is finite.
\end{definition}
In the following of the article, $\theta$-stability and $\theta$-semistability as well as the notations $X^{\theta\text{-s}}$ and $X^{\theta\text{-ss}}$  always refer to this definition.
\subsection{Correspondence between linear characters and elements in the center of the Lie algebra of $H$}\label{subsect_correspondence_character}
The Lie algebra of $G$ is denoted by  $\mathfrak{g}$ and the real Lie algebra of $H$ is $\mathfrak{h}$. Fix a non-degenerate scalar product $\left\langle\dots,\dots\right\rangle$ on $\mathfrak{h}$ invariant under the adjoint action.
\begin{proposition}[Polar decomposition]
For all $g\in G$ there exists a unique $(h,Y)\in H\times \mathfrak{h}$ such that $g=h\exp(i Y)$ such an expression is called a polar decomposition. This implies for the Lie algebra $\mathfrak{g}=\mathfrak{h}\oplus i\mathfrak{h}$.
\end{proposition}
\begin{proof}
It follows from \cite{lie_groups_III} Theorem 6.6.
\end{proof}

The first step in Kempf-Ness theory is to associate to a character $\chi^\theta\in\mathcal{X}^*(G)$ an element in the center $Z(\mathfrak{h})$ of the Lie algebra $\mathfrak{h}$. As $H$ is compact, its image under a complex character lies in the unit circle. Consider the differential of the character at the identity, it is a $\C$-linear map $d\chi^\theta_{\Id}:\mathfrak{g}\to \C$. The inclusion $\chi^\theta(H)\subset S^1$ implies for the Lie algebra $d\chi^\theta_{\Id}(\mathfrak{h})\subset i \RR$. By $\C$-linearity, $d\chi^\theta_{\Id}(i \mathfrak{h})\subset \RR$ and the following map is $\RR$-linear
\begin{equation}
    \map{d\chi^\theta_{\Id}(i\dots)}{\mathfrak{h}}{\RR}{Y}{d\chi^\theta_{\Id}(i Y)}.
\end{equation}
The invariant scalar product on $\mathfrak{h}$ identifies this linear form with an element of $\mathfrak{h}$ denoted by  $\theta$ satisfying for all $Y\in\mathfrak{h}$
\[
\left\langle \theta ,  Y \right\rangle = i d\chi^\theta_{\Id}( Y ).
\]
Moreover, as the scalar product is invariant for the adjoint action and so is the character $\chi^{\theta}$, the element $\theta$ lies in the center of $\mathfrak{h}$. This construction is $\ZZ$-linear so that it extends to an $\RR$-linear map 
\begin{equation*}
\begin{array}{ccccccc}
      \iota & : & \mathcal{X}^* (G)^\RR & \to & Z(\mathfrak{h}) \\
       & & \chi^\theta & \mapsto & \theta
\end{array}
\end{equation*}
\begin{proposition}
The $\RR$-linear map $\iota$ is an isomorphism from $\mathcal{X}^*(G)^{\RR}$ to $Z(\mathfrak{h})$.
\end{proposition}
\begin{proof}
As $G$ is a complex reductive group $G=Z(G) \mathcal{D}(G)$ with $Z(G)$ its center and $\mathcal{D}(G)$ its derived subgroup. Then $\mathcal{X}^*(G)$ identifies with the set of linear characters of the torus $Z(G)$. Hence $\mathcal{X}^*(G)$ is a $\ZZ$-module of rank the complex dimension of $Z(G)$ so that $\dim_{\RR} \mathcal{X}^*(G)^{\RR} = \dim_{\RR} Z(\mathfrak{h})$. It remains to prove that $\iota$ is injective. Let $\chi^{\theta}$ a linear character such that $d\chi^\theta_{\Id}(i Y)=0$ for all $Y\in \mathfrak{h}$. By $\C$-linearity and polar decomposition $d\chi^\theta_{\Id}=0$. Hence for any $g\in G$ the differential at $g$ is also zero $d\chi^\theta_{g}=0$. As $G$ is connected, $\chi^{\theta}$ is the trivial character.
\end{proof}
\begin{remark}
This isomorphism justifies the notation $\chi^{\theta}$ for elements in $\mathcal{X}^*(G)^{\RR}$, such elements are uniquely determined by a choice of $\theta\in Z(\mathfrak{h})$, moreover 
\[
\chi^{\theta}\chi^{\theta'}=\chi^{\theta+\theta'}.
\]
\end{remark}
\subsection{Correspondence between symplectic quotient and GIT quotient}
\begin{definition}[Real moment map]\label{def_moment}
The real moment map $\mu : X\to \mathfrak{h}$ is defined thanks to the invariant scalar product $\left\langle\dots,\dots\right\rangle$ by
\begin{equation*}
    \left\langle \mu(x),Y\right\rangle=\frac{1}{2}\left.\frac{d}{dt} ||\exp (i t Y).x||^2\right|_{t=0}
\end{equation*}
for all $Y\in\mathfrak{h}$ and $x\in X$. In this section the real moment map is just called the moment map. Later on complex and hyperkähler moment maps are also considered.
\end{definition}
\begin{example}\label{example_torus}
Assume the compact group $H$ is a torus $T$. The ambient space decomposes as an orthogonal direct sum $W=\bigoplus_{\chi^{\alpha}} W_{\chi^{\alpha}}$ with $\chi^{\alpha}$ linear characters of $T$ and
\[
W_{\chi^{\alpha}} = \left\lbrace x\in W \left|  t.x=\chi^{\alpha} (t) w \text{ for all } t\in T  \right.  \right\rbrace
\]
Similarly to \ref{subsect_correspondence_character},  a character $\chi^{\alpha}$ is uniquely determined by an element $\alpha$ in $\mathfrak{t}$ the Lie algebra of $T$ such that
\[
i d\chi_{\Id}^{\alpha}( Y ) = \left\langle \alpha , Y \right\rangle.
\]
Let $A$ the finite subset of elements $\alpha\in\mathfrak{t}$ such that $W_{\chi^{\alpha}} \ne \left\lbrace0\right\rbrace$. Let us compute $\mu_T$ the moment map for the torus action. Let $x = \sum_{\alpha\in A} x_{\chi^{\alpha}}$ in $W$, for $Y$ in $\mathfrak{t}$ the Lie algebra of $T$
\begin{eqnarray*}
\left\langle \mu_T (x) , Y \right\rangle & = &  \frac{1}{2}\left.\frac{d}{dt} ||\exp (i t Y).x||^2\right|_{t=0}  \\
& = &  \sum_{\alpha\in A} i d\chi^{\alpha}_{\Id}(Y)   \left|\left| x_{\chi^{\alpha}}  \right|\right|^2 \\
& = &  \left\langle \sum_{\chi\in A}\left|\left| x_{\chi^{\alpha}}\right|\right|^2  \alpha  , Y \right\rangle 
\end{eqnarray*}
Therefore the non-degeneracy of the scalar product implies  $\mu_T (x)=\sum_{\chi\in A}\left|\left| x_{\chi^{\alpha}}\right|\right|^2  \alpha$. In particular the image of $\mu_T$ is the cone $C(A)\subset \mathfrak{t}$ spanned by positive coefficients combinations of elements $\alpha\in A$. This example proves to be useful later on.
\end{example}
\begin{proposition}[Guillemin-Sternberg \cite{Guillemin}]\label{prop_surj}
$d_x \mu$ the differential of the moment map at $x$ is surjective if and only if the stabilizer of $x$ in $H$ is finite.
\end{proposition}
\begin{proof}
A computation using the definition of the moment map and the symplectic form gives for $v\in T_x X$ a tangent vector at $x$ and $Y\in\mathfrak{h}$
\[
\left\langle d_x\mu (v), Y  \right\rangle = \omega\left(\left.\frac{d}{d t} \exp(t Y).x\right|_{t=0}, v \right).
\]
This relation is often taken as a definition of the moment map. By non degeneracy of the symplectic form $\omega$ it implies that $Y$ is orthogonal to the image of $d_x \mu$ if and only if the stabilizer of $x$ contains $\exp(t Y)$ for all $t\in\RR$. Hence the differential of the moment map is surjective if and only if the stabilizer of $x$ is finite.
\end{proof}

\begin{lemma}\label{prop_critical_moment}
Let $\chi^\theta\in\mathcal{X}^*(G)^\RR$ and $x\in X$, then $\phi^{\theta,x}$ has a minimum at the identity if and only if $\mu(x)=\theta$. 

Moreover if $\phi^{\theta,x}$ has a minimun at the identity and at a point $h \exp(i Y)$ with $h\in H$ and $Y\in\mathfrak{h}$, then $\exp(i Y).x=x$. 
\end{lemma}
\begin{proof}
Up to a shift in the definition of the moment map, this result is \cite[Corollary A.7]{migliorini}. The proof is recalled as it is useful for next proposition.

For all $h\in H$ and $g\in G$ \[\phi^{\theta,x}(h g)=\phi^{\theta,x}(g)\]
so that the differential of $\phi^{\theta,x}$ at the identity vanishes on $\mathfrak{h}$. For $Y'+ i Y\in\mathfrak{h}\oplus i \mathfrak{h}$ this differential is
\begin{eqnarray*}
d\phi^{\theta,x}_{\Id} (Y'+i Y)=d\phi^{\theta,x}_{\Id} (i Y) &=&  \left.\frac{d}{dt} ||\exp (i t Y).x||^2\right|_{t=0} - d\chi^\theta_{\Id}(i Y)-\overline{d\chi^\theta_{\Id}(i Y)} \\
&=& 2\left\langle\mu(x),Y\right\rangle - 2\left\langle\theta,Y\right\rangle.
\end{eqnarray*}
last equality follows from the definition of the moment map $\mu$ and the discussion in \ref{subsect_correspondence_character} defining $\theta$ and proving the reality of $d\chi^\theta_{\Id}(i Y)$.

So far we proved that $\phi^{\theta,x}$ has a critical point at the identity if and only if $\mu(x)=\theta$, it remains to prove that this critical point is necessarily a minimum. Let $\phi^{\theta,x}$ be critical a the identity and $g\in G$ written in polar form $g=h\exp(i Y)$. The action of $i Y$ is hermitian so that it can be diagonalized in an orthonormal basis $(e_j)$ such that $i Y.e_j=\lambda_j e_j$ with $\lambda_j \in \RR$.
\begin{eqnarray*}
    \phi^{\theta,x}(h\exp(i Y))-\phi^{\theta,x}(\Id) =& \phi^{\theta,x}(\exp(i Y))-\phi^{\theta,x}(\Id)& \\
    =& \sum_j \left|\exp(\lambda_j) p(e_j,x)\right|^2 &-\log\left(\prod_j \exp(2 r_j\lambda_j)\right)\\
    & & -\sum_j \left|p(e_j,x)\right|^2 
\end{eqnarray*}
with $r_j$ real parameters determined by $\chi^\theta\in \mathcal{X}^*(G)^\RR$. As $\phi^{\theta,x}$ is critical at the identity:
\begin{equation*}
    0=\left.\frac{d}{dt} \phi^{\theta,x}\left(exp(i t Y)\right)\right|_{t=0}=\sum_j\left(2\lambda_j \left|p(e_j,x)\right|^2 -2 r_j  \lambda_j\right).
\end{equation*}
Combining the two previous equations
\begin{equation*}
    \phi^{\theta,x}(h\exp(i Y))-\phi^{\theta,x}(\Id) = \sum_j (\exp(2\lambda_j)-2\lambda_j -1)\left|p(e_j,x)\right|^2.
\end{equation*}
So that $\phi^{\theta,x}(h\exp(i Y))-\phi^{\theta,x}(\Id)\ge 0$ with equality if and only if $\exp(i Y).x=x$. Hence when $\phi^{\theta,x}$ has a critical point at the identity, it is necessarily a minimum.

\end{proof}
\begin{proposition}\label{corollary_orbit}
Let $\chi^\theta\in\mathcal{X}^*(G)^\RR$ then $\mu\inv(\theta)\subset X^{\theta\text{-ss}}$. Moreover, a point $x_0$ is $\theta$-stable if and only if the orbit $G.x_0$ intersects $\mu\inv(\theta)$ exactly in a $H$-orbit.
\end{proposition}
\begin{proof}
First statement follows from definition of stability \ref{def_real_stab} and Lemma \ref{prop_critical_moment}.

Assume $x_0$ is $\theta$-stable, then its orbit is closed in $X^{\theta\text{-ss}}$ and  $G.x_0\cap\mu\inv(\theta)$ is not empty. Let $x$ lies in this intersection, then $\phi^{\theta,x}$ has a minimum at the identity. For all $g,g'\in G$
\begin{equation*}
    \phi^{\theta,g.x}(g')=\phi^{\theta,x}(g'g)+\log\left(\left|\chi^\theta(g)\right|^2\right)
\end{equation*}
Hence $\phi^{\theta,g.x}(g')$ is minimum for $g'=g\inv$.
Now if $g\in G$ verifies $g.x\in\mu\inv(\theta)$ by Lemma \ref{prop_critical_moment}, $\phi^{\theta,g.x}(g')$ has a minimum not only at $g'=g\inv$ but also at the identity. By the second statement of previous lemma, $g\inv=h\exp(i Y)$ with $h\in H$ and $\exp(i Y).x=x$. As $x$ is stable, its stabilizer is finite so that $\exp(i Y)=\Id$ and $g\inv\in H$. Moreover for any $h\in H$, the map $\phi^{\theta,h.x}$ has a minimum at identity hence $h.x\in\mu\inv(\theta)$ so that $G.x_0 \cap \mu\inv(\theta)=H.x$.

Conversely suppose $G.x_0 \cap \mu\inv(\theta)=H.x $. First $x_0$ is $\theta$-semistable. By Migliorini \cite[Proposition A.9]{migliorini}, the orbit $G.x_0$ is closed in $X^{\theta\text{-ss}}$. It remains to prove that the stabilizer of $x_0$ is finite. By Lemma \ref{prop_critical_moment} the map $\phi^{\theta,x}$ is minimum at the identity. Let $Y\in\mathfrak{h}$ such that $\exp(i Y)$ is in the stabilizer of $x$. Then $\left|\chi^{\theta}\left(\exp(i Y)\right)\right|=1$, otherwise either $\phi^{\theta,x}\left(\exp( i Y)\right)<\phi^{\theta,x}(\Id)$ or $\phi^{\theta,x}\left(\exp(- i Y)\right)<\phi^{\theta,x}(\Id)$. Hence $\phi^{\theta,x}(\exp(i Y))=\phi^{\theta,x}(\Id)$ and $\exp(i Y)\in H$ so that $Y=0$ and the stabilizer of $x$ is finite.
\end{proof}

\begin{remark}\label{remark_kempf_bij}
For $\chi^{\theta}\in\mathcal{X}^*(G)$ such that $\theta$-stability and $\theta$-semistability coincide. Last proposition implies that the inclusion $\mu\inv(\theta)\subset X^{\theta\text{-ss}}$ goes down to a continuous bijective map 
\begin{equation*}
\mu\inv(\theta)/H \xrightarrow{\sim} X^{\theta\text{-ss}} // G.
\end{equation*} 
This result is a particular instance of Kempf-Ness theory, it gives a natural bijection between a symplectic quotient and a GIT quotient. Hoskins \cite{Hoskins} proved that this map is actually an homeomorphism.
\end{remark}

\subsection{Hilbert-Mumford criterion for stability}\label{subsect_HM}
Next theorem is a variant of the usual Hilbert-Mumford criterion for stability. It applies to real parameters $\chi^{\theta}\in \mathcal{X}^*(G)^{\RR}$ not only to to linear characters. Instead of algebraic one-parameter subgroups it relies on one-parameter real Lie groups defined for $Y\in\mathfrak{h}$ by
\[
\begin{array}{ccc}
    \RR & \to & G  \\
     t &\mapsto  & \exp(i t Y)
\end{array}
\]
Many variants of Hilbert-Mumford criterion for one-parameter real Lie groups are given in \cite{GVR_Hilbert}. Before proving the criterion, two classical technical lemmas are necessary.
\begin{lemma}\label{lemma_log_chi}
Let $\chi^{\theta}\in \mathcal{X}^*(G)^{\RR}$ and $Y\in\mathfrak{h}$, for $t\in \RR$
\[
\log\left|\chi^{\theta}\left(\exp(i t Y )\right)\right|^2= 2 \left\langle \theta, Y\right\rangle t.
\]
\end{lemma}
\begin{proof}
We prove it for $\chi^{\theta}\in \mathcal{X}^*(G)$ and deduce for elements in $\mathcal{X}^*(G)^{\RR}$ by $\RR$-linearity.
\begin{eqnarray*}
\left.\frac{d}{d t}\right|_{t=s} \log\left|\chi^{\theta}\left(\exp(i t Y )\right)\right|^2 &= & \frac{1}{ \left|\chi^{\theta}\left(\exp(i s Y )\right)\right|^2} \left.\frac{d}{d t}\right|_{t=s}\left|\chi^{\theta}\left(\exp(i t Y )\right)\right|^2 \\
& = &\left.\frac{d}{d t}\right|_{t=s}\left|\chi^{\theta}\left(\exp(i (t-s) Y )\right)\right|^2 \\
& = & \left.\frac{d}{d t}\right|_{t=0}\left|\chi^{\theta}\left(\exp(i t Y )\right)\right|^2 \\
& = & 2 d \chi^{\theta}_{\Id}(i Y)
\end{eqnarray*}
By the construction of the element $\theta\in Z(\mathfrak{h})$ from \ref{subsect_correspondence_character} we conclude that
\[
\left.\frac{d}{d t}\right|_{t=s} \log\left|\chi^{\theta}\left(\exp(i t Y )\right)\right|^2=  2\left\langle\theta,Y\right\rangle
\]
and
\[
\log\left|\chi^{\theta}\left(\exp(i t Y )\right)\right|^2 = 2 \left\langle \theta, Y\right\rangle t .
\]
\end{proof}

\begin{lemma}\label{lemma_f_Z}
Let $x_0\in X^{\theta\text{-s}}$ such that $\phi^{\theta,x_0}$ is minimum at the identity. Let $Z\in\mathfrak{h}$ and decompose $x_0$ in a basis of eigenvectors of the hermitian endomorphism $i Z$
\[
x_0=\sum_{\lambda} x^0_{\lambda}
\]
with
\[
\exp(i Z) x^0_{\lambda} = \exp(\lambda) x^0_{\lambda}.
\]
Then either $\left\langle \theta, Z\right\rangle <0$ or there exists $\lambda>0$ with $x^0_{\lambda}\ne 0$.
\end{lemma}
\begin{proof}
By Lemma \ref{prop_critical_moment} and Proposition \ref{corollary_orbit}, as $x_0$ is $\theta$-stable, the Kempf-Ness map $\phi^{\theta,x_0}$ reaches its minimum exactly on $H$. For $Z\in\mathfrak{h}$ consider the map $f_Z$ defined for $t$ real by
\begin{equation*}
f_Z (t) = \phi^{\theta,x_0}\left(\exp(i Z t) \right).
\end{equation*}
$f_Z$ reaches its minimum only at $t=0$. We can compute $f_Z(t)$ using the decomposition of $x_0$ in eigenvectors of $i Z$ and Lemma \ref{lemma_log_chi}
\begin{equation} \label{f_Z}
f_Z(t) = \sum_{\lambda} \exp(2 t\lambda) \left|\left|x_{\lambda}^0\right|\right|^2-2\left\langle\theta,Z\right\rangle t .
\end{equation}
Its second derivative is 
\[
f''_Z(t)=\sum_{\lambda} 4 \lambda^2 \exp(2 t\lambda) \left|\left|x_{\lambda}^0\right|\right|^2.
\]
Then $f_Z$ is convex, moreover it reaches its minimum only at $t=0$ so that 
\[
\lim_{t\to+\infty} f_Z(t) = +\infty.
\]
Looking at equation \eqref{f_Z} this implies either $\left\langle \theta,Z\right\rangle<0$ or there exists $\lambda>0$ with $x^0_{\lambda}\ne 0$.
\end{proof}

\begin{theorem}[Hilbert-Mumford criterion for stability]\label{theorem_HM}
Let $\theta\in\mathcal{X}^*(G)^{\RR}$ and $x\in X$. The following statements are equivalent
\begin{enumerate}[label=(\roman*)]
\item\label{item:1}$x$ is $\theta$-stable.
\item  For all $Y\in\mathfrak{h}$, different from zero, such that $\lim_{t\to+\infty}\exp(i t Y).x$ exists  then $\left\langle\theta,Y\right\rangle < 0$.
\end{enumerate}
\end{theorem}
\begin{proof}

\underline{not $(i)$ implies not $(ii)$}

Let $x\in X\setminus X^{\theta\text{-s}}$. Then if $\phi^{\theta,x}$ admits a minimum, the stabilizer of $x$ is not finite and this minimum is reached on an unbounded subset of $G$. Thus there exists an unbounded minimizing sequence for $\phi^{\theta,x}$. By polar decomposition and $H$ invariance we can assume it has the following form $\left(\exp i Y_n \right)_{n\in\mathbb{N}}$ with $(Y_n)_{n\in\mathbb{N}}\in \mathfrak{h}^{\mathbb{N}}$ unbounded. The hermitian space $W$ admits an orthonormal basis $B^n=\left(e^n_1,\dots,e^n_d\right)$ made of eigenvectors of $i Y_n$ with associated eigenvalues $\lambda^n_1,\dots , \lambda^n_d$.
\[
\exp(i Y_n).e^n_k=\exp(\lambda^n_k) e^n_k.
\]
This basis allows to compute:
\[
\phi^{\theta,x}\left(\exp i Y_n\right)=\sum_{k=1}^d \exp\left(2 \lambda^n_k \right) \left|\left|x^n_k\right|\right|^2 - 2 \left\langle \theta, Y_n\right\rangle
\]
with $x_k^n=p(x,e^n_k)e^n_k$ the components of $x$ in the basis $B^n$. By compactness of the set of orthonormal frames, we can assume the sequence of basis $(B^n)_{n\in \NN}$ converges to an orthonormal basis $B=\left(e_1,\dots,e_k\right)$. Let $x_k=p(x,e_k)e_k$ the components of $x$ in the basis $B$. Then $\lim_{n\to +\infty} x^n_k = x_k$. Let
\[
\Sigma_n =\sum_{k=1}^d \left|\lambda_k^n\right| 
\]
As $(Y_n)_{n\in\NN}$ is unbounded, up to an extraction of a subsequence, we can assume that $\lim_{n\to +\infty} \Sigma_n=+\infty $ and that the following limit exist and are finite:
\[
Y:=\lim_{n\to +\infty} \frac{Y_n}{\Sigma_n}
\]
and
\[
\lambda_k:=\lim_{n\to +\infty} \frac{\lambda_k^n}{\Sigma_n}.
\]
Now one can bound from bellow the values $\phi^{\theta,x}\left(\exp i Y_n\right)$ of the minimizing sequence
\begin{eqnarray*}
\phi^{\theta,x}\left(\exp i Y_n\right) &\ge& \sum_{\left\lbrace k\left|x_k\ne 0\right.\right\rbrace} \exp\left(2 \lambda^n_k \right) \left|\left|x^n_k\right|\right|^2  -2 \left\langle \theta, Y_n\right\rangle.  \\
&\ge & \sum_{\left\lbrace k\left|x_k\ne 0\right.\right\rbrace} \exp\left( 2\left(\lambda_k+o(1)\right)\Sigma_n \right) \left(\left|\left|x_k\right|\right|^2+o(1)\right) \\
& & -  2 \left( \left\langle \theta, Y\right\rangle +o(1) \right) \Sigma_n
\end{eqnarray*}
with $o(1)$ some sequences going to zero when $n$ goes to infinity. As the left-hand side is the value of a minimizing sequence, it cannot go to plus infinity. Hence $\left\langle \theta, Y\right\rangle \ge 0$, moreover if $x_k\ne 0$ Then $\lambda_k\le 0$. We conclude as $Y$ satisfies $\lim_{t\to+\infty} \exp(i t Y ).x$ exists and $\left\langle \theta, Y\right\rangle \ge 0$.

\underline{$(i)$ implies $(ii)$}

Let $x\in X^{\theta\text{-s}}$, by Lemma \ref{prop_critical_moment} and Proposition \ref{corollary_orbit} there exists $g_0\in G$ such that for $x_0=g_0.x$, the Kempf-Ness map $\phi^{\theta,x_0}$ reaches its minimum exactly on $H$. Now let $Y\in\mathfrak{h}$ such that $\lim_{t\to +\infty}\exp(i t Y).x$ exists then $\lim_{n\to +\infty }\exp(i n Y).x$ exists. For all $n\in \NN$ polar decomposition provides unique $h_n\in H$ and $Z_n\in\mathfrak{h}$ such that
\begin{equation*}\label{eq_Z_n}
\exp(i n Y)=h_n \exp( i Z_n ) g_0.
\end{equation*}
Then $Z_n$ is unbounded. Proceed as in the first part of the proof, $i Z_n$ is an hermitian endomorphism denote by $\lambda_1^n,\dots,\lambda_d^n$ its eigenvalues and let
\[
\Sigma_n=\sum_{k=1}^d \left|\lambda_k^n\right|.
\]
We can assume that $\lim_{n\to +\infty} \Sigma_n=+\infty$ and that the following limits exist and are finite:
\[
Z:=\lim_{n\to +\infty} \frac{Z_n}{\Sigma_n}
\]
and
\[
\lambda_k:=\lim_{n\to +\infty} \frac{\lambda_k^n}{\Sigma_n}.
\]
Then denoting by $x_k^0$ the components of $x_0$ in an orthonormal basis of eigenvectors of $i Z$
\begin{eqnarray*}
\phi^{\theta,x}\left(\exp(i Z_n)g_0\right) & \ge & \sum_{\left\lbrace k \left|x_k\ne 0\right.\right\rbrace} \exp\left(2 \left(\lambda_k +o(1)\right)\Sigma_n\right)\left(\left|\left|x_k\right|\right|^2+o(1)\right) \\
& & - 2\left(\left\langle \theta, Z\right\rangle+o(1)\right)\Sigma_n + \log\left|\chi^{\theta}(g_0)\right|^2
\end{eqnarray*}
By Lemma \ref{lemma_f_Z} either $\left\langle\theta,Z\right\rangle<0$ or there exists $\lambda_k>0$ with $x^0_k\ne 0$. In any case
\[
\lim_{n\to +\infty} \phi^{\theta,x}\left(\exp(i Z_n)g_0\right) = +\infty.
\]
Then the relation $\eqref{eq_Z_n}$ defining $Z_n$ implies
\begin{equation}\label{eq_limit_phi}
\lim_{n\to +\infty} \phi^{\theta,x}(\exp(in Y))=+\infty.
\end{equation}
Decompose $x$ in a basis of eigenvectors of the hermitian endomorphism $i Y$
\[
x=\sum_{\lambda} x_{\lambda}
\]
then
\[
\phi^{\theta,x}(\exp(i n Y) ) = \sum_{\lambda} \exp(2 n \lambda) \left|\left| x_{\lambda}\right|\right|^2 - 2 \left\langle\theta, Y\right\rangle n.
\]
As the limit $\lim_{n\to +\infty}\exp(i n Y).x$ is assumed to exist, $\lambda\le 0$ if $x_{\lambda}\ne 0$. Then the condition \eqref{eq_limit_phi} implies $\left\langle\theta, Y\right\rangle<0$.

\end{proof}
\subsection{Regular locus}\label{subsect_regular_locus}

In this subsection the closed subvariety $X$ is not relevant, the action of $G$ and $H$ on the ambient hermitian vector space $W$ is studied. First note that the moment map can be defined not only on $X$ but on the whole space $W$. Let $T\subset H$ a maximal torus. As in Example \ref{example_torus} the ambient space $W$ decomposes as an orthogonal direct sum $W=\bigoplus W_{\chi^{\alpha}}$ with $\chi^{\alpha}$ characters of $T$ and
\[
W_{\chi^{\alpha}} = \left\lbrace x \in W \left|   t.x = \chi^{\alpha}(t) x \text{ for all }t\in T \right. \right\rbrace.
\]
Denote by $A$ the finite subset of elements $\alpha\in\mathfrak{t}$ such that for the character $\chi^{\alpha}$ the space $W_{\chi^{\alpha}}$ is not zero then
\[
W=\bigoplus_{\alpha\in A} W_{\chi^{\alpha}}.
\]
As before the link between linear characters and elements in $\mathfrak{t}$ is through the invariant pairing $\left\langle\dots,\dots\right\rangle$
\[
i d\chi^{\alpha}_{\Id} (\beta) = \left\langle \alpha, \beta\right\rangle.
\]
Hence if $\beta$ is orthogonal to the $\RR$ vector space spanned by $A$
\[
\chi^{\alpha}(\exp t\beta)=1
\]
for all $\alpha \in A$ so that $\exp t\beta$ is in the kernel of the action of $H$ on $W$. From the beginning this kernel is assumed to be trivial, hence the vector space spanned by $A$ is $\mathfrak{t}$. As in Example \ref{example_torus}, the image of $\mu_T$, the moment map relative to the $T$-action, is the cone spanned by positive combinations of $A$. For any $A'$ finite subset of $\mathfrak{t}$ the cone spanned by positive combinations of $A'$ is:
\[
C(A'):=\left\lbrace \sum_{\alpha\in A'} a_{\alpha} \alpha \left| \quad a_{\alpha} \ge 0 \text{ for all }\alpha\in A' \right. \right\rbrace.
\]
For any $\beta\in \mathfrak{t}$
\begin{eqnarray*}
\left\langle \mu(x), \beta\right\rangle & = & \frac{d}{d t} \left. \left|\left| \exp ( i t \beta ).x\right|\right|^2 \right|_{t=0}  \\
 & = & \left\langle \mu_T(x), \beta \right\rangle.
\end{eqnarray*}
Hence, as noted by Kirwan \cite{Kirwan}, if $\mu(x)\in\mathfrak{t}$ then $\mu(x)=\mu_T (x)$. For $A'$ a finite subset of $\mathfrak{t}$ we denote by $\dim A'$ the dimension of the vector space spanned by $A'$. 
\begin{lemma}
Let $x\in W$ such that for all $A'\subset A$ with $\dim A'<\dim \mathfrak{t}$, the value of the moment map $\mu_T(x)$ does not lie in $C(A')$. Then the stabilizer of $x$ is finite.
\end{lemma}
\begin{proof}
Decompose $x$ according to its weight $x=\sum_{\alpha\in A} x_{\alpha}$ then
\[
\mu_T(x)=\sum \left|\left|x_{\alpha}\right|\right|^2 \alpha.
\]
Denote by $A_x$ the set of elements $\alpha$ such that $x_{\alpha}\ne 0$. The hypothesis about $\mu_T(x)$ implies that $\dim A_x=\dim\mathfrak{t}$. Now for $\beta\in\mathfrak{t}$
\[
\exp ( \beta t ).x =\sum_{\alpha\in A_x} \chi^{\alpha}(\exp \beta t) x_{\alpha}.
\]
Hence if $\exp \beta t$ is in the stabilizer of $x$, for all $\alpha\in A_x$ the pairing with $\beta$ vanishes $\left\langle \alpha,\beta\right\rangle=0$. As $A_x$ spans $\mathfrak{t}$ this implies that $\beta=0$ and the stabilizer of $x$ in $T$ is finite.
\end{proof}
Previous lemma justifies the introduction of the following nonempty open subset of $\mathfrak{t}$
\[
C(A)^{\reg}:=C(A)\setminus \bigcap_{\substack{ A'\subset A  \\ \dim A'< \dim \mathfrak{t}  }  } C(A').
\]
As all maximal torus of $H$ are conjugated, the set $C(A)^{\reg} \cap Z(\mathfrak{h})$ is independent of a choice of maximal torus $T$.
\begin{proposition}\label{proposition_B_reg}
For $\theta\in C(A)^{\reg} \cap Z(\mathfrak{h})$, every $\theta$-semistable points are $\theta$-stable, $W^{\theta\text{-ss}}=W^{\theta\text{-s}}$ and in particular $X^{\theta\text{-ss}}=X^{\theta\text{-s}}$.
\end{proposition}
\begin{proof}
Let $x\in W^{\theta\text{-ss}}$, then $\overline{G.x}$ meets $\mu\inv(\theta)$. But $\overline{G.x}\setminus G.x$ is a union of $G$-orbits of dimension strictly smaller than $G.x$, points in those orbits has stabilizer with dimension greater than one. By previous lemma every point in $\mu\inv(\theta)$ has a finite stabilizer. Thus $G.x\cap \mu\inv(\theta)\ne\emptyset$ and the stabilizer of $x$ is finite so that $x$ is $\theta$-stable.
\end{proof}
Kirwan \cite{Kirwan}, Ness-Mumford \cite{Ness}, Sjamaar \cite{Sjamaar}, Harada-Wilkin \cite{Harada} and Hoskins \cite{Hoskins} studied a stratification of $W$. It relies on the Morse theory of the following map. For $\theta\in Z(\mathfrak{h})$
\[
\map{h_{\theta}}{W}{\RR}{x}{\left|\mu(x) - \theta \right|^2}
\]
with $\left|\dots\right|$ the norm defined by the invariant pairing $\left\langle\dots,\dots\right\rangle$ on $\mathfrak{h}$. A critical point of a smooth map $f$ is a point $x$ where the differential vanishes $d_x f=0$. A critical value of $f$ is the image $f(x)$ of a critical point $x$. The gradient of $h_{\theta}$ is the vector field defined  thanks to the hermitian pairing $p(\dots,\dots)$ for $x\in W$ and $v\in T_x W$ by
\[
p\left(\grad_x h_{\theta}, v \right) =  d_x h_{\theta} . v
\]
For $x\in W$ the negative gradient flow relative to $h_{\theta}$ is the map
\[
\map{\gamma^{\theta}_{x}}{\RR^{\ge 0}}{ W } {t}{\gamma^{\theta}_x(t)}
\]
uniquely determined by the condition
\[
\left.\frac{d \gamma^{\theta}_x (s)}{ds}  \right|_{s=t} = -\grad_{\gamma^{\theta}_x (t)} h_{\theta}.
\]
and $\gamma^{\theta}_x (0)= x$. By \cite{Sjamaar} and \cite{Harada} it is well defined and for any $x$ the limit $\lim_{t\to +\infty} \gamma^{\theta}_x (t) $ exists and is a critical point of $h_{\theta}$. $S^{\theta} $ is the set of point $x\in W$ with negative gradient flow for $h_{\theta}$ converging to a point where $h_{\theta}$ reaches its minimal value $0$:
\[
S^{\theta} := \left\lbrace x \in W \left| \lim_{t\to +\infty } \gamma^{\theta}_x (t) \in \mu \inv (\theta) \right.\right\rbrace.
\]
This is the open strata of the stratification, Sjamaar called it the set of analitically semistable points. When the stability parameter is a true character i.e. $\chi^{\theta}\in\mathcal{X}^*(G)$, Hoskins \cite{Hoskins} proved that this strata coincides with the $\theta$-semistable locus. Here we want to consider any $\chi^{\theta}\in\mathcal{X}^*(G)^{\RR}$, the proof of the inclusion $S^{\theta}\subset W^{\theta\text{-ss}}$ is the same and it is enough for our purpose.
\begin{proposition}\label{prop_S0}
$S^{\theta} $ is a subset of $W^{\theta\text{-ss}}$.
\end{proposition}
\begin{proof}
The flow $\gamma_x (t)$ belongs to the orbit $G.x$ hence $\lim_{t\to +\infty } \gamma_x (t)\subset \overline{G.x}$. Therefore if $x\in S^{\theta} $ then $\overline{G.x}\cap \mu\inv(\theta) \ne \emptyset$.
\end{proof}
An important feature of the map $h_{\theta}$ is that its critical points lie in a finite union $\bigcup_{A'\subset A} \mu\inv\left(H.\beta(A',\theta)\right)$ indexed by the subsets of the finite set $A$. With $\beta(A',\theta)$ the projection of $\theta$ to the closed convex $C(A')$ and $H.\beta(A',\theta)$ the adjoint orbit of $\beta(A',\theta)$.

\begin{lemma}\label{lemma_h_inv}
By definition of the projection to a closed convex in an euclidian space $\left| \beta(A',\theta) -\theta \right|$ is the distance between $\theta$ and the cone $C(A')$, define 
\begin{equation}\label{eq_d_theta}
   d_{\theta} = \inf_{\substack{A'\subset A \\ \beta(A',\theta)\ne \theta }} \left| \beta(A',\theta) -\theta \right|^2
\end{equation}
then $d_{\theta}>0$ and ${h_{\theta}}\inv \left[0,d_{\theta} \right[ \subset S^{\theta} $.
\end{lemma}
\begin{proof}
For any $h\in H$ by invariance of the scalar product under the adjoint action and as $\theta\in Z(\mathfrak{h})$
\[
\left|h.\beta(A',\theta)-\theta\right|^2=\left|\beta(\theta,A')-\theta \right|^2.
\]
Hence if $x$ is a critical point of $h_{\theta}$ not in $\mu\inv(\theta)$, then $x\in\mu\inv (H.\beta(A',\theta))$ for some $\beta(A',\theta)$ different from $\theta$ and
\[
\left| \mu(x) - \theta  \right|^2=\left|\beta(\theta,A') - \theta\right|^2 > d_{\theta}.
\]
So that the only critical value of $h_{\theta_0}$ in the intervalle $\left[0,d_{\theta} \right[$ is $0$.

Now for any $x\in W$, the map $t\mapsto h_{\theta}\left(\gamma^{\theta}_x(t)\right)$ can only decrease, and it converges to a critical value. Therefore if $x\in h_{\theta}\inv\left[0,d_{\theta} \right[$ the negative gradient flow converges necessarily to a point $\lim_{t\to +\infty } \gamma^{\theta}_x (t)$  which belongs to $ h_{\theta}\inv (0)=\mu\inv(\theta)$ so that $x\in S^{\theta} $.
\end{proof}

\begin{theorem}\label{theorem_C_open}
Let $\theta_0\in C(A)^{\reg} \cap Z(\mathfrak{h})$, there is an open neighborhood $V_{\theta_0}$ of $\theta_0$ in $C(A)^{\reg} \cap Z(\mathfrak{h})$ such that for all $\theta\in V_{\theta_0}$, $\theta$-stability and $\theta_0$-stability coincide $W^{\theta_0\text{-ss}}=W^{\theta\text{-ss}}$.
\end{theorem}
\begin{proof}
Let $\epsilon>0$ such that $B(\theta_0,\epsilon)$ the ball of center $\theta_0$ and radius $\epsilon$ in $\mathfrak{t}$ is included in $C(A)^{\reg}$. Then when $\theta$ varies in $B(\theta_0,\epsilon)$ it does not meet any frontier of a cone $C(A')$ with $A'\subset A$. So that for $\theta\in B(\theta_0,\epsilon)$, for all $A'\subset A$, $\beta(\theta,A')\ne 0$ if and only if $\beta(\theta_0,A')\ne 0$. Thus the subset indexing the infima defining $d_{\theta}$ and $d_{\theta_0}$ in \eqref{eq_d_theta} are identical. As the projection to closed convex is a continuous map, the map $\theta\mapsto d_{\theta}$ is continuous on $B(\theta_0,\epsilon)$. 
Therefore one can chose $\epsilon'>0$ such that
\begin{itemize}
\item $d_{\theta}>\tfrac{d_{\theta_0}}{2}$ for all $\theta\in B(\theta_0,\epsilon')$.
\end{itemize}
Moreover $\epsilon'$ can be chosen to satisfy the following conditions
\begin{itemize}
\item $B(\theta_0,\epsilon') \subset C(A)^{\reg}$
\item $\epsilon'^2<\frac{d_{\theta_0}}{2}$
\end{itemize}
Let $\theta$ in $B(\theta_0,\epsilon')\cap Z(\mathfrak{h})$, we shall see that $W^{\theta\text{-ss}}=W^{\theta_0\text{-ss}}$. First note that
$\theta\in C(A)^{\reg} \cap Z(\mathfrak{h})$ and Proposition \ref{proposition_B_reg} implies $W^{\theta\text{-ss}}=W^{\theta\text{-s}}$.

For $x\in W^{\theta\text{-ss}}=W^{\theta\text{-s}}$, by Proposition \ref{corollary_orbit} there exists $g\in G$ such that $g.x\in\mu\inv(\theta)$. Then $\left|\mu(g.x)-\theta_0\right|^<\frac{d_{\theta_0}}{2}$ and $g.x\in h_{\theta_0}\inv\left[0,d_{\theta_0} \right[$. By Lemma \ref{lemma_h_inv}, $g.x\in S^{\theta_0} $ and by Proposition \ref{prop_S0} $g.x$ is $\theta_0$-semistable so that $x\in  W^{\theta_0\text{-ss}}$.

Similarly for $x\in W^{\theta_0\text{-ss}}$, there exists $g\in G$ such that $g.x\in\mu\inv(\theta_0)$. Then $\left|\mu(g.x) - \theta \right|^2<\tfrac{d_{\theta_0}}{2} $ and as $\tfrac{d_{\theta_0}}{2}<d_{\theta}$, the point $g.x$ lies in $h_{\theta}\inv\left[0,d_{\theta} \right[$ therefore $x$ is $\theta$-stable.
\end{proof}
Considering again the closed subvariety $X\subset W$ one defines the regular locus:
\begin{definition}[Regular locus]
The regular locus $B^{\reg}$ is the set of elements $\theta\in C(A)^{\reg} \cap Z(\mathfrak{h})$ such that for all $x\in X^{\theta\text{-ss}}$ the stabilizer of $x$ in $G$ is trivial and $X^{\theta\text{-ss}}\ne\emptyset$. 
\end{definition}
\begin{proposition}
The regular locus $B^{\reg}$ is the union of some connected components of $C(A)^{\reg} \cap Z(\mathfrak{h})$.
\end{proposition}
\begin{proof}
By Theorem \ref{theorem_C_open}, if $\theta$ and $\theta'$ are in the same connected component of $C(A)^{\reg} \cap Z(\mathfrak{h})$ then $W^{\theta\text{-ss}}=W^{\theta'\text{-ss}}$. Hence if $\theta\in C(A)^{\reg} \cap Z(\mathfrak{h})$ is such that for all $x\in X^{\theta\text{-ss}}$ the stabilizer of $x$ in $G$ is trivial and $X^{\theta\text{-ss}}\ne\emptyset$, the same holds for $\theta'$ in the same connected component of $C(A)^{\reg} \cap Z(\mathfrak{h})$.
\end{proof}

\begin{remark}
Note that the regular locus $B^{\reg}$ can be empty, for instance if the center $Z(\mathfrak{h})$ is a subset of a cone $C(A')$ with $\dim A'<\dim \mathfrak{t}$. Fortunately it is non-empty for the application to Nakajima quiver varieties of next sections.
\end{remark}
In next subsection we prove that the real moment map is a locally trivial fibration over the regular locus $B^{\reg}$.

\subsection{Trivialization of the real moment map over the regular locus}\label{subsect_triv_moment}
Next construction follows ideas from Hitchin-Karlhede-Lindström-Roček and is illustrated in \cite[Figure 3 p.348]{hklr}.
\begin{proposition}\label{remark_y_theta_x}
For $\chi^\theta\in\mathcal{X}^*(G)^\RR$ and $x$ a $\theta$-stable point with trivial stabilizer, there exists a unique $Y^{\theta,x}\in\mathfrak{h}$ such that $\exp\left(i Y^{\theta,x}\right).x\in\mu\inv(\theta)$. Moreover for $h\in H$ the adjoint action of $h$ on $Y^{\theta,x}$ satisfies
\begin{equation}\label{equivariant_Y}
    h .Y^{\theta,x} = Y^{\theta,h.x}.
\end{equation}
Let $\theta'=\mu(x)$ and $x'=\exp\left(i Y^{\theta,x}\right).x$, then
\begin{equation}\label{inverse_Y}
    Y^{\theta',x'}=-Y^{\theta,x}.
\end{equation}.
\end{proposition}
\begin{proof}
As $x$ is $\theta$-stable, by Proposition \ref{corollary_orbit} the orbit $G.x$ intersects $\mu\inv (\theta)$ exactly on a $H$-orbit. There exists $g\in G$ such that $g.x\in\mu\inv(\theta)$. Apply polar decomposition to this element $g=h_0 \exp\left(i Y^{\theta,x}\right)$ with $h_0\in H$ and $Y^{\theta,x}\in\mathfrak{h}$. Then 
\begin{equation*}
    \mu\inv(\theta)\cap G.x = H. \exp\left(i Y^{\theta,x} \right).x
\end{equation*}
Take $Y'$ such that $\exp\left(i Y'\right).x\in\mu\inv (\theta)$ then 
\[
\exp(i Y').x=h \exp(i Y^{\theta,x}).x
\]
for some $h$ in $H$. By triviality of the stabilizer of $x$ and uniqueness of polar decomposition $Y'=Y^{\theta,x}$ hence $Y^{\theta,x}$ is uniquely determined.
Let us check $H$-equivariance, for $h\in H$
\begin{equation*}
    \mu\inv(\theta)\ni h \exp(i Y^{\theta,x}).x=\exp\left(i  h . Y^{\theta,x}\right).h.x
\end{equation*}
by uniqueness $Y^{\theta,h.x}= h . Y^{\theta,x}$. Equation \eqref{inverse_Y} is clear.
\end{proof}  

\begin{remark}
The assumption that $x$ has a trivial stabilizer can be relaxed. Then there exists $Y^{x,\theta}\in \mathfrak{h}$ such that
\[
\left\lbrace Y \in \mathfrak{h} \left| \right. \exp(i Y).x\in \mu\inv(\theta)  \right\rbrace = \left(\Stab_H x \right). Y^{\theta,x}
\]
The right-hand side is the orbit of $Y^{\theta,x}$ under the adjoint action of the stabilizer of $x$ in $H$. For applications to quiver varieties we only need to consider the case of a trivial stabilizer.
\end{remark}

\begin{lemma}\label{lemma_diff}
Let $\theta\in Z(\mathfrak{h})$ and $x_0$ a $\theta$-stable point with trivial stabilizer. There exists an open neighborhood $U_{\theta,x_0}$ of $(\theta,x_0)$ in $\mathfrak{h}\times X$ and a smooth  map
\begin{equation*}
    \map{Y}{U_{\theta,x_0}}{\mathfrak{h}}{(\theta',x')}{Y(\theta',x')}
\end{equation*}
such that $\mu\left(\exp\left(i {Y}(\theta',x')\right).x'\right)=\theta'$.

\end{lemma}
\begin{proof}
Note that when $\theta\in Z(\mathfrak{h})$ necessarily ${Y}(\theta,x)$ is equal to the $Y^{\theta,x} $ introduced in previous proposition. Let $Y^{\theta,x_0}$ such that $x:=\exp\left(i Y^{\theta,x_0}\right).x_0$ is in the intersection $G.x_0\cap \mu\inv(\theta)$. Consider the map
\begin{equation*}
    \map{f}{\mathfrak{h}\times\mathfrak{h}\times X}{\mathfrak{h}}{(Y',\theta',x')}{\mu\left(\exp(i Y').x'\right)-\theta'}
\end{equation*}
in order to use the implicit function theorem on a neighborhood of $\left(Y^{\theta,x_0},\theta,x_0\right)$ we first prove that the differential of $f$ with respect to $Y'$ at $(Y^{\theta,x_0},\theta,x_0)$ is invertible. As $x$ has a finite stabilizer, the embedding of tangent spaces $T_{x} H.x\xhookrightarrow{} T_{x} G.x$ identifies with the embedding 
\begin{equation}\label{eq_tangent_space}
    \mathfrak{h}\cong T_{\Id} H\xhookrightarrow{ }T_{\Id} G\cong \mathfrak{h}\oplus i\mathfrak{h}.
\end{equation}
By Proposition \ref{prop_surj}, $d\mu$ is surjective so that $\mu\inv(\theta)$ is a smooth manifold and $\ker d\mu_{x}=T_{x} \mu\inv(\theta)$. Proposition \ref{corollary_orbit} implies $\mu\inv(\theta)\cap G.x= H.x$ Restricting $d_x \mu $ to the tangent space of the $G$-orbit we obtain the following short exact sequence
\begin{equation*}
0\xhookrightarrow{} T_x H.x \xhookrightarrow{} T_x G.x \xrightarrow{\left.d_x \mu \right|_{T_x G.x}}  \mathfrak{h}\xrightarrow{} 0.
\end{equation*}
the surjectivity follows from dimension counting and the identification of the tangent spaces with \eqref{eq_tangent_space}. Thus we obtain the expected invertibility of the differential with respect to $Y'$ of $f$ at $(Y^{\theta,x_0},\theta,x_0)$, the map $d_{Y'}f_{(Y^{\theta,x_0},\theta,x_0)}$, identifies with an invertible map $i\mathfrak{h}\to \mathfrak{h}$. The implicit function theorem applies and gives the existence of $U_{\theta,x_0}\subset \mathfrak{h}\times X$ an open neighborhood of $(\theta,x_0)$ and the expected smooth map $Y(\dots,\dots)$.

\end{proof}

Next theorem is a first result concerning local triviality of the moment map, over the regular locus $B^{\reg}$ the real moment map is a locally trivial fibration.
\begin{theorem}\label{theorem_triviality}
Let $\theta_0$ in $B^{\reg}$, and $U_{\theta_0}$ the connected component of $B^{\reg}$ containing $\theta_0$. There is a diffeomorphism $f$ such that the following diagram commutes

\begin{equation*}
\begin{tikzcd}
 U_{\theta_0} \times {\mu}\inv \left(\theta_0\right)\arrow[swap]{rd}{} \arrow[r,"f"',"\sim"] &  \mu\inv( U_{\theta_0} ) \arrow{d}{{\mu}_{}} \\
 &   U_{\theta_0} 
\end{tikzcd}
\end{equation*}
Moreover $f$ is $H$ equivariant so that the diagram goes down to quotient
\begin{equation*}
\begin{tikzcd}
 U_{\theta_0} \times {\mu}\inv \left(\theta_0\right)/H\arrow[swap]{rd}{} \arrow[r,"\sim"] &  \mu\inv( U_{\theta_0} )/H \arrow{d} \\
 &   U_{\theta_0} 
\end{tikzcd}
\end{equation*}

\end{theorem}
\begin{proof}
For $\theta\in U_{\theta_0}$ we know from \ref{subsect_regular_locus} that $X^{\theta\text{-s}}=X^{\theta_0\text{-s}}\ne\emptyset$. Define $f$ by
\begin{equation*}
    f(\theta,x):=\exp\left(i Y(\theta,x)\right).x
\end{equation*}
It follows from Proposition \ref{remark_y_theta_x} that it is invertible with inverse 
\begin{equation*}
    f\inv(x')=\left(\mu(x'),\exp\left(i Y(\theta_0,x')\right).x'\right).
\end{equation*}
Lemma \ref{lemma_diff} implies that $f$ is a diffeomorphism. Equivariance follows from equation \eqref{equivariant_Y} so that $f(\theta,h.x)=h.f(\theta,x)$  and $f$ goes down to a diffeomorphism between quotients.
\end{proof}

In next sections Nakajima quiver varieties are considered, they admit an additional hyperkähler structure. A similar trivialization is established in this hyperkähler context.

\section{Quiver varieties and stability} \label{section_quiver}
\subsection{Generalities about quiver varieties}\label{subsect_generalities}
The quiver varieties considered in this paper were introduced by Nakajima \cite{nakajima1994}.
Let $\Gamma$ be a quiver with vertices $\Omega_0$ and edges $\Omega_1$. For an edge $\gamma\in\Omega_1$ we denote $t(\gamma)\in \Omega_0$ its tail and $h(\gamma)\in \Omega_0$ its head, we define the reverse edge $\gammabar$ such that $t(\gammabar)=h(\gamma)$ and $h(\gammabar)=t(\gamma)$.
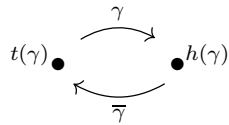
\begin{equation*}
    \begin{tikzcd}
    {}^{t(\gamma)}\bullet \arrow[r,bend left, "\gamma"] &\arrow[l,bend left, "\overline{\gamma}"] \bullet^ {h(\gamma)}
    \end{tikzcd}
\end{equation*}
Let $\overline{\Omega}_1:=\left\lbrace \gammabar \left|\gamma\in\Omega_1\right.\right\rbrace$ and $\widetilde{\Omega}:=\Omega_1\sqcup\overline{\Omega}_1$. For $\gammabar\in\overline{\Omega}_1$ we set $\overline{\gammabar}:=\gamma$ to obtain an involution on $\widetilde{\Omega}$.
The extended quiver $\widetilde{\Gamma}$ is obtained by adding an inverse to all edges in $\Omega_1$, its set of vertices is $\Omega_0$ and its set of edges is $\widetilde{\Omega}$. Let $\epsilon:\widetilde{\Omega}\to \left\lbrace-1,1\right\rbrace$ be the map
\[
\left\lbrace
\begin{array}{ccc}
     \epsilon(\gamma)= 1 &\text{ if }&\gamma\in\Omega_1  \\
     \epsilon(\gamma)= -1 &\text{ if }&\gamma\in\overline{\Omega}_1
\end{array}
\right.
\]
We fix a dimension vector $v\in \mathbb{N}^{\Omega_0}$. A representation of the quiver $\Gamma$ with dimension vector $v$ is a pair $(V,\phi)$ with $V=\bigoplus_{j\in \Omega_0}V_j$ a graded vector space with $\dim V_j=v_j$ and $\phi=(\phi_{\gamma})_{\gamma\in\Omega_1}$ a collection of linear maps $\phi_{\gamma}:V_{t(\gamma)}\to V_{h(\gamma)}$. A subrepresentation is a subspace $W\subset V$ with a compatible $\Omega_0$-grading and preserved by $\phi$. The set of quiver representations with dimension vector $v$ is identified with 
\[
\Rep\left(\Gamma,v\right):=\bigoplus_{\gamma\in\Omega_1}\Mat_{\C}(v_{h(\gamma)},v_{t(\gamma)}).
\]
For construction of quiver varieties it is interesting to consider representations of the extended quiver $\widetilde{\Gamma}$
\[
\Rep\left(\widetilde{\Gamma},v\right):=\bigoplus_{\gamma\in\widetilde{\Omega}}\Mat_{\C}(v_{h(\gamma)},v_{t(\gamma)}).
\]
It is a complex vector space, the complex structure considered in this section is
\[
I.(\phi_{\gamma} )_{\gamma\in\widetilde{\Omega}} = (i \phi_{\gamma} )_{\gamma\in\widetilde{\Omega}} 
\]
The group $\GL_v:=\prod_{i\in {\Omega_0}}\GL_{v_i}(\C)$ acts linearly on $\Rep(\widetilde{\Gamma},v)$
\[
g.\left(\phi_\gamma \right)_{\gamma\in\widetilde{\Omega}}:=\left(g_{h(\gamma)}\phi_\gamma g_{t(\gamma)}\inv
\right)_{\gamma\in\widetilde{\Omega}}.\]
The diagonal embedding of $\C^*$ in $\GL_v$ acts trivially so that the action goes down to an action of the group
\[
G_v:=\GL_v/\C^*,
\]
which identifies with
\[
G_v\cong\left\lbrace \left.(g_j)_{j\in {\Omega_0}}\in \GL_v\right| \prod_{j\in {\Omega_0}} \det(g_{j})=1 \right\rbrace.
\]
Note that $G_v$ is isomorphic to a product of a special linear group and a finite number of general linear groups so that it is a reductive group. The Lie algebra of $\GL_v$, respectively  $G_v$ is $\gl_v=\bigoplus_{j\in {\Omega_0}}\gl_{v_j}(\C)$ respectively.
\[
\gv = \left\lbrace (x_j)_{j\in {\Omega_0}}\in \gl_v \left|\sum_{j\in {\Omega_0}} \tr x_j = 0\right. \right\rbrace
\]
The center of $\gv$ is
\[
 Z(\gv)=\left\lbrace (\xi_j \Id_{v_j})_{j\in {\Omega_0}} \left|(\xi_j)_{j\in\Omega_0}\in {(\C)}^{\Omega_0} \text{ with } \sum_{j\in\Omega_0} v_j \xi_j =0 \right. \right\rbrace.
\]
Let $\theta\in\ZZ^{\Omega_0}$ such that $\sum_{j\in\Omega_0}v_j\theta_j=0$, define $\chi^\theta$ a character of $G_v$ by
\begin{equation} \label{chi_theta_quiver}
    \chi^\theta\left((g_j)_{j\in\Omega_0}\right)=\prod_{j\in\Omega_0}\det(g_j)^{-\theta_j}.
\end{equation}
The $\theta$-semistable locus, respectively $\theta$-stable locus in the sense of Mumford's Geometric Invariant Theory \cite{mumford_git}, are denoted by   $\Rep(\widetilde{\Gamma},v)^{\theta\text{-ss}}$, respectively $\Rep(\widetilde{\Gamma},v)^{\theta\text{-s}}$.

\begin{sect_definition}[Complex moment map]
The complex moment map is defined by
\begin{equation*}
    \map{\mu_{\C}}{\Rep(\widetilde{\Gamma},v)}{\gv}{\left(\phi_{\gamma}\right)_{\gamma\in\widetilde{\Omega}}}{\sum_{\gamma\in\widetilde{\Omega}}\epsilon(\gamma)\phi_{\gamma}\phi_{\overline{\gamma}}  }
\end{equation*}
it is $G_v$-equivariant for the adjoint action on $\gv$.
\end{sect_definition}
This complex moment map will be related to the real moment map of Definition \ref{def_moment} in next section.
\begin{sect_definition}[Nakajima quiver variety]\label{def_quiver_varieties}
For $\xi\in Z(\gv)$, the set $\mu_{\C}\inv(\xi)$ is an affine variety in $\Rep(\widetilde{\Gamma},v)$, it inherits a $G_v$ action. Nakajima quiver varieties are defined as GIT quotients:
\begin{equation*}
    \mathcal{M}_v^\theta(\xi):=\mu_{\C}\inv(\xi)\cap\Rep(\widetilde{\Gamma},v)^{\theta\text{-ss}}//G_v.
\end{equation*}
\end{sect_definition}
Those varieties are interesting from the differential geometry point of view and have an hyperkähler structure. We are interested in the family formed by those varieties when the parameters $\xi$ and $\theta$ are varying.
\subsection{King's characterization of stability of quiver representations}\label{subsect_King}
As in Section \ref{sect_kempf} the geometric invariant theory has a symplectic counterpart.
$\Rep\left(\widetilde{\Gamma},v\right)$ is an hermitian vector space with norm
\[
\left|\left|(\phi_{\gamma})_{\gamma\in\widetilde{\Omega}} \right|\right|^2 = \sum_{\gamma\in\widetilde{\Omega}} \tr(\phi_{\gamma}\phi_{\gamma}^{\dagger}).
\]
The $G_v$-action restricts to a unitary action of the maximal compact subgroup
\[
U_v=\left\lbrace \left.(g_j)_{j\in {\Omega_0}}\in \prod_{j\in {\Omega_0}} U_{v_j}\right| \prod_{j\in {\Omega_0}} \det(g_{v_j})=1 \right\rbrace
\]
The Lie algebra of $U_v$ is
\[
\uv = \left\lbrace (x_j)_{j\in {\Omega_0}}\in \bigoplus_{j\in {\Omega_0}}\mathfrak{u}_{v_j} \left|\sum_{j\in {\Omega_0}} \tr x_j = 0\right. \right\rbrace
\]
with $U_{v_j}$, respectively $\mathfrak{u}_{v_j}$, the group of unitary matrices, respectively the space of skew-hermitian matrices of size $v_j$. The real moment map $\mu_I$ for the $U_v$ action satisfies
\begin{equation*}
    \left\langle \mu_I(x),Y\right\rangle =\frac{1}{2} \left.\frac{d}{dt} ||\exp ( i t. Y).x||^2\right|_{t=0}
\end{equation*}
for $Y\in \uv$. The pairing is defined for $Y=(Y_j)_{j\in\Omega_0}$ and $Z=(Z_j)_{j\in\Omega_0}$ by
\begin{equation}\label{pairing_invariant}
    \left\langle Y, Z\right\rangle = \sum_{j\in\Omega_0} \tr(Y_j Z_j).
\end{equation}
As in \ref{subsect_correspondence_character}, to the character $\chi^{\theta}$ defined by \eqref{chi_theta_quiver} is associated the following element of the Lie algebra $\uv$
\begin{equation}\label{theta_uv}
\theta = (-i \theta_{j} \Id_{v_j})_{j\in {\Omega_0}}\in \uv.
\end{equation}
Indeed for $Y=(Y_j)_{j\in\Omega_0}$ in the Lie algebra $\uv$, by the usual differentiation of the determinant map at identity
\begin{equation*}
    d\chi^{\theta}_{\Id}(i Y_j)=-\sum_{j\in\Omega_0} i \theta_{j} \tr(Y_j) =\left\langle \theta, Y \right\rangle.
\end{equation*}

We recall here an important result from King giving a characterization of $\theta$-stability for quiver representations.

\begin{sect_theorem}[King \cite{king} Proposition 3.1] \label{king_stab}
Let $\theta\in\mathbb{Z}^{\Omega_0}$ such that $\sum \theta_j v_j = 0 $ and $\chi^{\theta}$ the associated character defined by \eqref{chi_theta_quiver} .
\begin{enumerate}
\item A quiver representation $(V,\phi)\in\Rep\left(\widetilde{\Gamma},v\right)$ is $\theta$-semistable if and only if for all subrepresentation $W\subset V$
\[
\sum_{j\in \Omega_0} \theta_{j} \dim W_j \le 0.
\]
\item  A quiver representation $(V,\phi)$ is a $\theta$-stable if and only if for all subrepresentation $W$ different from $0$ and $(V,\phi)$
\[
\sum_{j\in \Omega_0} \theta_{j} \dim W_j < 0.
\]
\end{enumerate} 
\end{sect_theorem}

The symplectic point of view allows to consider real parameters $\theta\in\RR^{\Omega_0}$ such that $\sum_{j\in\Omega_0}v_j\theta_j=0$. They are associated to elements $\chi^{\theta}\in\mathcal{X}^*(G_v)^{\RR}$ with well-defined modulus:
\[
\left|\chi^{\theta}\left((g_j)_{j\in\Omega_0}\right)\right|=\prod_{j\in\Omega_0} \left|\det(g_j)\right|^{-\theta_j}.
\]
The set of $\theta$-stable points in $\Rep\left(\widetilde{\Gamma},v\right)$ is defined by Definition \ref{def_real_stab}. The end of this section is devoted to a generalization of the second point of King's theorem for real parameters $\theta\in \RR^{\Omega_0}$ such that $\sum \theta_j v_j =0$.

Let $Y=\left( Y_j\right)_{j\in\Omega_0}\in \uv$, the $i Y_j$ are hermitian endomorphisms of $V^j$. For $\lambda\in \RR$ denote by $V^j_{\le \lambda}$ the subspace of $V^j$ spanned by eigenvectors of $i Y_j$ with eigenvalues smaller than $\lambda$ then define
\[
V_{\le\lambda}:=\bigoplus_{j \in \Omega_0} V^j_{\le\lambda}.
\]
\begin{sect_lemma}\label{lemma_subrep}
Let $x=(V,\phi)$ in $\Rep\left(\widetilde{\Gamma},v\right)$ and $Y\in \uv$. The limit 
\[
\lim_{t\to+\infty} \exp(i t Y).x
\]
exists if and only if for every $\lambda$ real, $V_{\le\lambda}$ defines a subrepresentation of $(V,\phi)$
\end{sect_lemma}
\begin{proof}
For all $j\in\Omega_0$ take a basis of $V^j$ formed by eigenvectors of $i Y_j$ and assume the eigenvalues repeated according to multiplicities are ordered
\[
\lambda^j_{1}\le \lambda^j_{2} \le \dots \le \lambda^j_{v_j} .
\]
In those basis of eigenvectors, for $\gamma\in \widetilde{\Gamma}$ one can write the matrix of $\phi^{\gamma}$ and compute the action of $\exp(i t Y)$
\[
\left( \exp(i t Y) .\phi \right)_{\gamma} = 
\begin{pmatrix}
\phi^{\gamma}_{1,1} & e^{ t(\lambda^{h(\gamma)}_1-\lambda^{t(\gamma)}_2)} \phi^{\gamma}_{1,2} & \dots & \\
e^{ t(\lambda^{h(\gamma)}_2-\lambda^{t(\gamma)}_1)}\phi^{\gamma}_{2,1} & & &  \\
\vdots &\dots & e^{ t(\lambda^{h(\gamma)}_a-\lambda^{t(\gamma)}_b)}  \phi^{\gamma}_{a,b} & \dots \\
e^{t(\lambda^{h(\gamma)}_{v_{h(\gamma)}}-\lambda^{t(\gamma)}_1)}\phi^{\gamma}_{v_{h(\gamma)},1} & & \dots & 

\end{pmatrix}
\]
the limit exists if and only if the matrix is upper triangular i.e. $\phi(V_{\le\lambda}) \subset V_{\le\lambda}$ and $V_{\le \lambda}$ defines a subrepresentation of $(V,\phi)$.
\end{proof}

Next result is the generalization of King's theorem relative to $\theta$-stability of quiver representations for a real parameter $\theta$. Its proof relies on previous lemma and the Hilbert-Mumford criterion for real one-parameter Lie groups \ref{theorem_HM}.
\begin{theorem}\label{th_king_lie}
Let $\theta\in \RR^{\Omega_0}$ such that $\sum_{j\in\Omega_0}\theta_j v_j = 0$ and $\chi^{\theta}$ the associated element in $\mathcal{X}^*(G_v)^{\RR} $. A quiver representation $(V,\phi)$ is $\theta$-stable if and only if for all subrepresentation $W\subset V$ different from $0$ and $(V,\phi)$
\[
\sum_{j\in\Omega_0} \theta_j \dim W_j < 0.
\]
\end{theorem}
\begin{proof}
Let $x=(V,\phi)$ in $\Rep\left(\widetilde{\Gamma},v\right)^{\theta\text{-s}}$ a $\theta$-stable point. By Hilber-Mumford criterion (Theorem \ref{theorem_HM}), for all $Y\in \uv$ such that $\lim_{t\to +\infty} \exp(i t Y).x$ exists then $\left\langle\theta,Y\right\rangle < 0$.

Let $W$ be a subrepresentation  of $(V,\phi)$ different from $0$ and $(V,\phi)$. For all $j\in\Omega_0$ define $Y_j$ in $\mathfrak{u}_{v_j}$ such that $W_j$ is an eigenspace of $i Y_j$ with eigenvalue $\lambda_1$ and $W_j^{\bot}$ the orthogonal complement of $W_j$ is an eigenspace of $i Y_j$ with eigenvalue $\lambda_2$ and $\lambda_2>\lambda_1$. By previous lemma
$\lim_{t\to +\infty} \exp(i t Y).x$ exists.
\begin{eqnarray*}
\left\langle\theta,Y\right\rangle &=& -\sum_{j\in\Omega_0} \theta_j \left(\lambda_1 \dim W_j+\lambda_2\left(\dim V_j -\dim W_j \right)\right) \\
&=& - \sum_{j\in\Omega_0} (\lambda_1-\lambda_2) \theta_j \dim W_j
\end{eqnarray*}
because $\sum \theta_j v_j = 0$. Then Hilbert-Mumford criterion implies $\left\langle\theta,Y\right\rangle<0$, hence $\sum_{j\in\Omega_0} \theta_j \dim W_j <0$.

Conversely let $x=(V,\phi)$ a quiver representation such that for all subrepresentation $W\varsubsetneq V$ different from $0$
\[
\sum_{j\in\Omega_0} \theta_j \dim W_j < 0.
\]
Let $Y=\left( Y_j \right)_{j\in\Omega_0} \in \uv$ different from zero. The set of  eigenvalues of $i Y_j$ is ordered $\lambda^j_1<\dots<\lambda^j_{d_j}$. The set of all eigenvalues for all $j\in\Omega_0$ is also ordered
\[
\left\lbrace\lambda^j_k\right\rbrace_{\substack{j\in\Omega_0 \\ 1\le k \le d_j }} = \left\lbrace \lambda_1, \lambda_2, \dots, \lambda_m \right\rbrace
\]
with $\lambda_k <\lambda_{k+1}$. For convenience add an element $\lambda_0<\lambda_1$.
If $\lim_{t\to =\infty} \exp(i t Y).x$ exists, by previous lemma $V_{\le \lambda}$ is a subrepresentation of $(V,\phi)$. Moreover
\begin{eqnarray*}
\left\langle\theta,Y\right\rangle &=& -  \sum_{j\in\Omega_0} \theta_j \sum_{k=1}^{d_j} \lambda^j_k \left(\dim V^j_{\le\lambda^j_k}-\dim V^j_{\le\lambda^j_{k-1}}\right)\\
&=& -  \sum_{j\in\Omega_0} \theta_j \sum_{k=1}^m \lambda_k \left(\dim V^j_{\le\lambda_k}-\dim V^j_{\le\lambda_{k-1}}\right) \\
&=& -  \sum_{j\in\Omega_0} \theta_j \sum_{k=1}^{m-1} \left(\lambda_k-\lambda_{k+1}\right) \dim V^j_{\le \lambda_k}  \\
& & - \lambda_m \sum_{j\in\Omega_0} \theta_j \dim V^j_{\le \lambda_m}.
\end{eqnarray*} 
The last summand vanishes as $\sum \theta_j v_j = 0$,
\[
\left\langle\theta,Y\right\rangle = -  \sum_{k=1}^m \left(\lambda_k-\lambda_{k+1}\right) \sum_{j\in\Omega_0} \theta_j \dim V^j_{\le \lambda_k}
\]
As $Y\ne 0$, it has at least two distinct eigenvalues. Then $V_{\le \lambda_1}$ is a subrepresentation different from zero and $V$ and 
\[
-(\lambda_0-\lambda_1)\sum_{j\in\Omega_0} \theta_j \dim V^j_{\le \lambda_1}<0
\]
so that $\left\langle\theta,Y\right\rangle<0$.
\end{proof}
This result is useful in next section to characterize a regular locus for the hyperkähler moment map.

\section{Nakajima quiver varieties as hyperkähler quotients and trivialization of the hyperkähler moment map}\label{sect_hyperkahler_quiver}
After some reminder about the hyperkähler structure of Nakajima quiver varieties, trivializations of the hyperkähler moment map are discussed. 
\subsection{Hyperkähler structure on the space of representations of an extended quiver}\label{subsect_hk_rep}
The space $\Rep\left(\widetilde{\Gamma},v\right)$ is endowed with three complex structures
\begin{eqnarray*}
I.\left(\phi_\gamma,\phi_{\overline \gamma}\right)&=&(i \phi_\gamma, i \phi_{\overline \gamma}) \\
J.\left(\phi_\gamma,\phi_{\overline \gamma}\right)&=& (-\phi_{\overline \gamma}^\dagger,\phi_\gamma^\dagger)\\
K.\left(\phi_\gamma,\phi_{\overline \gamma}\right)&=& (-i\phi_{\overline \gamma}^\dagger,i\phi_\gamma^\dagger)
\end{eqnarray*}
satisfying quaternionic relations
\begin{equation}\label{quaternionic}
I^2=J^2=K^2=IJK=-1
\end{equation} 
and a norm
\[
\left|\left|(\phi_\gamma)_{\gamma\in\widetilde{\Omega}}\right|\right|^2=\sum_{\gamma\in\widetilde{\Omega}}\tr\left(\phi_\gamma \phi_\gamma^\dagger\right).
\]
For each complex structure, polarisation identity defines an hermitian pairing compatible with $||\dots||$. For example the hermitian pairing compatible with the complex structure $I$ used in previous section is
\begin{equation*}
    p_I \left(u,v\right)=\frac{1}{4}\left(||u+v||^2-||u-v||^2+i||u+I.v||^2 -i||u-I.v||^2\right)
\end{equation*}
$p_J(\dots,\dots)$ and $p_K(\dots,\dots)$ are similarly defined.
One expression is particularly simple
\[
p_I\left((\phi_\gamma)_{\gamma\in \widetilde{ \Omega}},(\psi_\gamma)_{\gamma\in \widetilde{ \Omega}}\right)= \sum_{\gamma\in\widetilde{\Omega}} \tr(\phi_\gamma \psi_\gamma^\dagger).
\]
\begin{remark}
Even if the hermitian metric relies on the choice of complex structure, by the polarisation identity the real part remains the same, it is the hyperkähler metric 
\begin{equation*}
    g(\dots,\dots):=\Ree p_I(\dots,\dots)=\Ree p_J(\dots,\dots)=\Ree p_K(\dots,\dots).
\end{equation*}
\end{remark}
\begin{definition}[Real symplectic forms]
As in equation \eqref{symp_herm} we define a real symplectic form for each complex structure
\begin{eqnarray*}
\omega_I(\dots,\dots)&:=&g(I\dots,\dots) \\
\omega_J(\dots,\dots)&:=&g(J\dots,\dots) \\
\omega_K(\dots,\dots)&:=&g(K\dots,\dots)
\end{eqnarray*}
\end{definition}
\begin{notations}
$I$-linear means $\C$-linear with respect to the complex structure $I$ and similarly for $J$-linear and $K$-linear.
\end{notations}
\begin{proposition}[Permutation of complex structures]\label{permutation_complex}
Consider the map
\begin{equation*}
    \map{\Psi}{\Rep\left(\widetilde{\Gamma},v\right)}{\Rep\left(\widetilde{\Gamma},v\right)}{x}{\frac{1}{2}\left(1+I+J+K\right).x}
\end{equation*}
It is an isomorphism from the hermitian vector space $\Rep\left(\widetilde{\Gamma},v\right)$ with the complex structure $I$ and hermitian pairing $p_I$ to the hermitian vector space $\Rep\left(\widetilde{\Gamma},v\right)$ with the complex structure $J$ and pairing $p_J$.

More generally it cyclically permutes the three complex structure $I,J,K$
\begin{equation}\label{psi_complex}
\begin{array}{ccc}
\Psi(I.x)&=&J.\Psi(x)\\
\Psi(J.x)&=&K.\Psi(x)\\
\Psi(K.x)&=&I.\Psi(x).
\end{array}
\end{equation}
Such a map is sometimes called an hyperkähler rotation.
\end{proposition}
\begin{proof}
Relations \eqref{psi_complex} follow from a computation with the quaternionic relations \eqref{quaternionic}. To prove the compatibility with the hermitian structures it is enough to check that $||\Psi(x)||=||x||$.
\begin{equation*}
\left|\left|(1+I+J+K).x\right|\right|^2 = g\left((1+I+J+K).x,(1+I+J+K).x\right).
\end{equation*}
The expected result is obtain after cancellations from the identity $g(I.u,u)=0$, similar relations for the other complex structures and quaternionic relations \eqref{quaternionic}.
\end{proof}

In \ref{subsect_generalities} an $I$-linear action of $G_v$ is described. The hyperkähler rotation $\Psi$ provides the following construction for $J$-linear and $K$-linear actions. This three actions coincide when restricted to the compact subgroup $U_v$.
\begin{definition}[Complexification of the action]
Thanks to polar decomposition, to define a linear action of $G_v$ compatible with the complex structure $J$ it is enough to define the action of $\exp(i.Y)$ for $Y\in\uv$. To highlight the complex structure used, this action is written $\exp(J.Y)\dots$ and defined by
\begin{equation*}
\exp(J.Y).x:=\Psi \left(\exp(i.Y).\Psi\inv (x)\right)
\end{equation*}
with the element $\exp(i.Y)$ of $G_v$ acting by the natural $I$-linear action previously described. Similarly
\begin{equation*}
    \exp(K.Y).x:=\Psi\inv \left(\exp(i.Y).\Psi (x)\right).
\end{equation*}
\end{definition}
\begin{remark}\label{remark_different_stability}
A point $x$ is $\theta$-(semi)stable with respect to the $I$-linear action if and only if $\Psi(x)$ is $\theta$-(semi)stable with respect to the $J$-linear action.
\end{remark}
\subsection{Hyperkähler structure and moment maps}
By Proposition \ref{permutation_complex} the various  $G_v$-actions previously described are compatible with the hermitian metrics so that the constructions of section \ref{sect_kempf} apply. They provide a moment map for each complex structure.
\begin{eqnarray*}
    \left\langle \mu_I(x),Y\right\rangle &=&\frac{1}{2} \left.\frac{d}{dt} ||\exp ( t.I. Y).x||^2\right|_{t=0}\\
    \left\langle \mu_J(x),Y\right\rangle &=&\frac{1}{2} \left.\frac{d}{dt} ||\exp ( t.J. Y).x||^2\right|_{t=0}\\
    \left\langle \mu_K(x),Y\right\rangle &=&\frac{1}{2} \left.\frac{d}{dt} ||\exp ( t.K. Y).x||^2\right|_{t=0}.
\end{eqnarray*}
The pairing is defined by \eqref{pairing_invariant}.

\begin{definition}[Hyperkähler moment map]\label{def_hk_moment}
Those three real moment maps fit together in an hyperkähler moment map $\mu_{\mathbb{H}}:\Rep\left(\widetilde{\Gamma},v\right)\to \uv\oplus\uv\oplus\uv$ defined by $\mu_{\mathbb{H}}=\left(\mu_I,\mu_J,\mu_K\right)$.
\end{definition}
The moment map $\mu_{\C}$ defined in  \ref{subsect_generalities} by
\begin{equation}
\mu_{\C} \left( (\phi_\gamma)_{\gamma\in\widetilde\Omega} \right) := \sum_{\gamma\in\widetilde{\Omega}} \epsilon(\gamma) \phi_\gamma \phi_{\overline \gamma}. \label{algebraic_moment_map}
\end{equation}
can be expressed from the real moment maps
\begin{equation*}
    \mu_{\C}:=\mu_J+i\mu_K.
\end{equation*}
it is a polynomial map with respect to the complex structure $I$.
\begin{remark}
By cyclic permutation of the complex structure, $\mu_K + i \mu_I$ is polynomial with respect to the complex structure $J$ and $\mu_I + i \mu_J$ is polynomial with respect to the complex structure $K$.
\end{remark}
 Take $(\theta_{J,j})_{j\in {\Omega_0}}$ and $(\theta_{K,j})_{j\in {\Omega_0}}$ in $\RR^{\Omega_0}$ such that $\sum_{j} v_j \theta_{J,j} = \sum_{j} v_j \theta_{K,j} = 0$. Associate to each of them an element in the center of the Lie algebra $\uv$
 \begin{eqnarray*}
 \theta_J&:=&\left(-i\theta_{J,j}\Id_{v_j}\right)_{j\in\Omega_0} \\
 \theta_K&:=&\left(-i\theta_{K,j}\Id_{v_j}\right)_{j\in\Omega_0}.
 \end{eqnarray*}
 Then $\theta_J+i\theta_K$ defines an element in the center of $\gv=\uv\oplus i\uv$. Hence $\mu_J\inv(\theta_J)\cap\mu_K\inv(\theta_K)=\mu_{\C}\inv(\theta_J+i\theta_K)$ is an affine variety embedded in the vector space $\Rep\left(\widetilde{\Gamma},v\right)$ endowed with the complex structure $I$ and stable under the $G_v$-action.
Section \ref{sect_kempf} does not apply directly to this situation as $\mu_{\C}\inv(\theta_J+i\theta_K)$ might be singular. However it applies to the action of $G_v$ on the ambiant space $\Rep\left(\widetilde{\Gamma},v\right)$. For $\theta_I\in\RR^{\Omega_0}$ such that $\sum_{j\in\Omega_0} v_j \theta_{I,j}=0$ consider the associated element $\chi^{\theta_I}\in \mathcal{X}^*(G_v)^{\RR}$.

\begin{definition}[Hyperkähler regular locus]
For $w\in \mathbb{N}^{\Omega_0}$ a  dimension vector
\begin{equation*}
    H_w:=\left\lbrace(\theta_I,\theta_J,\theta_K)\in\left(\RR^{\Omega_0}\right)^3\left|\sum_{j} w_j \theta_{I,j} = \sum_j w_j \theta_{J,j}= \sum_j w_j \theta_{K,j} =0\right.\right\rbrace.
\end{equation*}
The regular locus is
\begin{equation}
    H_v^{\reg}=H_v \setminus \bigcup_{w< v} H_w
\end{equation}
the union is over dimension vector $w\ne v$ such that $0\le w_i\le v_i$.
\end{definition}
\begin{remark}
This regular locus is empty unless the dimension vector $v$ is indivisible, then $H_v^{\reg}$ is the complementary of a finite union of codimension $3$ real vector space. 
\end{remark}

Thanks to Kempf-Ness theory, Nakajima quiver varieties can be constructed as hyperkähler quotients. The underlying manifold of the variety $\mathcal{M}_v^{\theta_I}\left(\theta_J +i \theta_K\right)$ (see definition \ref{def_quiver_varieties}) is :
\[
\mathfrak{m}_v(\theta_I,\theta_J,\theta_K) = \mu_{\mathbb{H}}\inv (\theta_I,\theta_J,\theta_K)/U_v
\]

\subsection{Trivialization of the hyperkähler moment map}\label{subsect_triviality_hk}
We study the family of Nakajima quiver varieties when the parameters $(\theta_I,\theta_J,\theta_K)$ are varying. Nakajima proved by consecutive uses of different complex structures that for $\theta$ and $\theta'$ in $H_v^{\reg}$ the manifolds $\mathfrak{m}_v(\theta_I,\theta_J,\theta_K)$ and $\mathfrak{m}_v(\theta_I',\theta_J',\theta_K')$ are diffeomorphic \cite[Corollary 4.2]{nakajima1994}. We use this idea of consecutive uses of different complex structures to prove that those manifolds fit in a locally trivial family over the regular locus $H_v^{\reg}$. First let us highlight relevant facts about the regular locus. 
\begin{lemma}
Let $\left(\theta_I,\theta_J,\theta_K\right)\in H_v^{\reg}$ and $x\in \mu_J\inv(\theta_J)\cap\mu_K\inv(\theta_K)$. Then $x$ is $\theta_I$-stable if and only if it is $\theta_I$-semistable.
\end{lemma}
\begin{proof}
If $x_0 \in \mu_{\mathbb{H}}\inv(\theta_I,\theta_J,\theta_K)$ its stabilizer in $G_v$ is trivial. Indeed Maffei proved that the differential of the moment map at $x_0$ is surjective \cite[Lemma 48]{Maffei}, then Proposition \ref{prop_surj} implies the triviality of the stabilizer of $x_0$. 

Let $x\in \mu_J\inv(\theta_J)\cap\mu_K\inv(\theta_K)$ a $\theta_I$-semistable point. Then $\overline{G_v .x}\cap \mu_I\inv(\theta_I)$ is not empty. As $\mu_J\inv(\theta_J)\cap\mu_K\inv(\theta_K)=\mu_{\C}\inv(\theta_J+i\theta_K)$ is $G_v$ stable, the closure of the orbit $\overline{G_v .x}$ meets $ \mu_{\mathbb{H}}\inv(\theta_I,\theta_J,\theta_K)$ at a point $x_0$. This point necessarily has a trivial stabilizer, hence $x_0\in G_v.x$ and $x$ is $\theta_I$-stable.
\end{proof}
Let $(\theta_I,\theta_J,\theta_K)\in H_v^{\reg}$ and consider first the complex structure $I$. By previous lemma and King's characterisation of stability (Theorem \ref{th_king_lie}), for $\theta'_I$ in an open neighborhood of $\theta_I$, stability with respect to $\theta'_I$ is the same as stability with respect to $\theta_I$.

Now consider the complex structure $J$. Thanks to Remark \ref{remark_different_stability} on the affine variety $\mu_K\inv(\theta_K)\cap\mu_I\inv(\theta_I)$ all $\theta_J$-semistable points are $\theta_J$-stable. Moreover for $\theta'_J$ in an open neighborhood of $\theta_J$, stability with respect to $\theta'_J$ is the same as stability with respect to $\theta_J$. Similarly for the complex structure $K$.

Assume that the dimension vector $v$ is a root of the quiver so that the moment map is surjective, see \cite[Theorem 2]{CB_surj}. Consider the diagram 
\begin{equation*}
\begin{tikzcd}
\mu_{\mathbb{H}}\inv \left(H_v^{\reg}\right)\arrow{r}{} \arrow{d}{} & \Rep\left(\widetilde{\Gamma},v\right)\arrow{d}{{\mu}_{\mathbb{H}}} \\
H_v^{\reg} \arrow[swap]{r}{} & \uv\oplus\uv\oplus\uv
\end{tikzcd}.
\end{equation*}.

\begin{theorem}[Local triviality of the hyperkähler moment map]\label{th_hk_moment}
Over the regular locus $H_v^{\reg}$, the hyperkähler moment map $\mu_{\mathbb{H}}$ is a locally trivial fibration compatible with the $U_v$-action:

Any $(\theta_I,\theta_J,\theta_K )\in H_v^{\reg} $ admits an open neighborhood $V$, and a diffeomorphism $f$ such that the following diagram commutes
\begin{equation*}
\begin{tikzcd}
V\times {\mu}_{\mathbb{H}}\inv (\theta_I,\theta_J,\theta_K )\arrow[swap]{rd}{} \arrow[r,"f"',"\sim"] &  \mu_{\mathbb{H}}\inv(V) \arrow{d}{\mu_{\mathbb{H}}} \\
 &  V
\end{tikzcd}
\end{equation*}
Moreover $f$ is compatible with the $U_v$-action so that the diagram goes down to quotient
\begin{equation*}
\begin{tikzcd}
V\times \mathfrak{m}_v(\theta_I,\theta_J,\theta_K)\arrow[swap]{rd}{} \arrow[r,"\sim"] &   \mu_{\mathbb{H}}\inv(V) /U_v \arrow[d] \\
 &  V
\end{tikzcd}
\end{equation*}

\end{theorem}
\begin{proof}
The method is similar to the proof of Theorem \ref{theorem_triviality} applied consecutively to the three complex structures. The idea of using different complex structures comes from \cite{nakajima1994} and \cite{kronheimer1989}. Take $(\theta_I,\theta_J,\theta_K)\in H_v^{\reg}$ and a connected open neighborhood $U_I\times U_J\times U_K$ such that for $\theta'_I\in U_I$, any $x\in\mu_J\inv(U_J)\cap\mu_K\inv(U_K)$ is $\theta_I'$-semistable if and only if it is  $\theta_I$-stable. Similarly for $U_J$ and $U_K$.  For any $x$ with $\mu_{\mathbb{H}}(x)=(\theta'_I,\theta'_J,\theta'_K)\in U_I\times U_J\times U_K$, by Proposition \ref{remark_y_theta_x} applied to the $I$-linear action of $G_v$ on $\Rep\left(\widetilde{\Gamma},v\right)$, there exists a unique $Y_I(\theta_I,x)\in \uv$ such that 
\[\exp\left(I.Y_I(\theta_I,x)\right).x\in \mu_{\mathbb{H}}\inv(\theta_I,\theta'_J,\theta'_K).\] 
Then by exchanging the three complex structures with hyperkähler rotations, there exists unique $Y_J(\theta_J,x)$ and $Y_K(\theta_K,x)$ such that 
\[\exp\left(J.Y_J(\theta_J,x)\right)\exp\left(I. Y_I(\theta_I,x)\right).x\in \mu_{\mathbb{H}}\inv(\theta_I,\theta_J,\theta'_K)\]
and 
\[\exp\left(K.Y_K(\theta_K,x)\right)\exp\left(J. Y_J(\theta_J,x)\right)\exp\left(I. Y_I(\theta_I,x)\right).x\in \mu_{\mathbb{H}}\inv(\theta_I,\theta_J,\theta_K).\]
This defines the map $f\inv$
\begin{equation*}
    f\inv(x):=\left((\theta'_I,\theta'_J,\theta'_K),\exp\left(K.Y_K(\theta_J,x)\right)\exp\left(J. Y_J(\theta_J,x)\right)\exp\left(I. Y_I(\theta_I,x)\right).x\right).
\end{equation*}
Lemma \ref{lemma_diff} implies the smoothness of $f\inv$. This map induces a diffeomorphism, indeed exchanging $\theta$ and $\theta'$ in previous construction produces the expected inverse 
\begin{equation*}
    f\left(x,(\theta'_I,\theta'_J,\theta'_K)\right):=\exp\left(I.Y_I(\theta'_I,x)\right)\exp\left(J.Y_J(\theta'_J,x)\right)\exp\left(K.Y_K(\theta'_K,x)\right).x
\end{equation*}
It follows from equation \eqref{inverse_Y} that the maps are inverse of each others. The exchange in the order of appearance of the complex structures $I,J$ and $K$ in the definition of $f$ and $f\inv$ are necessary as the exponentials do not necessarily commute.
The $U_v$-equivariance follows from equation \eqref{equivariant_Y}.
\end{proof}
Similarly one can consider the complex moment map $\mu_{\C}=\mu_J +i\mu_K$ instead of $\mu_{\mathbb{H}}$. The complex regular locus is $C_{v}^{\reg} :=C_v\setminus \bigcup_{w<v} C_w$ with
\[
C_w=\left\lbrace \xi\in \C^{\Omega_0} \left| \sum_{j\in\Omega_0} w_j \xi_j = 0 \right. \right\rbrace
\]

\begin{theorem}
The complex moment map is a locally trivial fibration over $C_v^{\reg}$. Any $\xi\in C_v^{\reg}$ admits an open neighborhood $V$, and a diffeomorphism $f$ such that the following diagram commutes
\begin{equation*}
\begin{tikzcd}
V\times {\mu}_{\mathbb{C}}\inv (\xi )\arrow[swap]{rd}{} \arrow[r,"f"',"\sim"] &  \mu_{\mathbb{C}}\inv(V) \arrow{d}{\mu_{\mathbb{C}}} \\
 &  V
\end{tikzcd}
\end{equation*}
\end{theorem}
\begin{proof}
The proof is similar to the hyperkähler situation.
\end{proof}
Denote $\pi:\mu_{\mathbb{H}}\inv(H_v^{\reg})/U_v\to H_v^{\reg}$ the map obtained taking the quotient of $\mu_{\mathbb{H}}$. Consider the cohomology sheaves $\mathcal{H}^i \pi_* \qlbar$ of the derived pushforward of the constant sheaf and the cohomology sheaves $\mathcal{H}^i \pi_! \qlbar$ of the derived compactly supported pushforward of the constant sheaf.
\begin{corollary}
The sheaves $\mathcal{H}^i \pi_* \qlbar$ and $\mathcal{H}^i p_! \qlbar$ are constant sheaves over $H_v^{\reg}$.
\end{corollary}
\begin{proof}
By Theorem \ref{th_hk_moment} those sheaves are locally constant. $H_v^{\reg}$ is a complementary of a finite union of codimension $3$ real vector spaces, hence it is simply connected so that the locally constant sheaves are constant.
\end{proof}
Nakajima explained to us that this corollary can also be obtained by generalizing Slodowy's construction \cite{Slodowy_four} to quiver varieties.

Finally we extend the trivialization of the hyperkähler moment map over lines constructed by Crawley-Boevey and Van den Bergh \cite{CB_VDB} using twistor spaces as told to us by Nakajima. 

Denote by $\mathbb{H}$, respectively $\mathbb{H}_0$, the set of quaternions, respectively  the set of purely imaginary quaternions and $\mathbb{H}_0^*=\mathbb{H}_0\setminus \left\lbrace 0 \right\rbrace$. The space $\mathfrak{u}_v^{\oplus 3}$ is identified with $\mathbb{H}_0\otimes_{\RR} \mathfrak{u}_v$. Then the hyperkähler moment map reads
\[
\mu_{\mathbb{H}} = I \otimes \mu_I + J \otimes \mu_J + K\otimes \mu_K.
\]
Once an orthonormal basis of $\RR^3$ is fixed, the triple of complex structures $I$, $J$ and $K$ is fixed and we write $\mu_{\mathbb{R}}=\mu_I$, $\mu_{\mathbb{C}}=\mu_J + i \mu_K$.
The hyperkähler moment map is assumed to be surjective and the dimension vector indivisible. Then $H_v^{\reg}$ is the open subset of generic parameters in $\mathbb{H}_0\otimes_{\RR} Z(\mathfrak{u}_v)$. For $\theta\in H_v^{\reg}$ a generic parameter and $S$ a contractible subset of $\mathbb{H}_0^*$, Crawley-Boevey and Van den Bergh constructed a trivialization of the hyperkähler moment map over $S\otimes\theta$, see \cite{CB_VDB} proof of Lemma 2.3.3 (in the statement of this lemma $S$ is chosen to be a complex line). The assumption contractible is relaxed in next theorem. It relies on the theory of twistor spaces developped by Penrose \cite{Penrose}, Atiyah-Hitchin-Singer \cite{AHSinger} and Salamon \cite{Salamon1982}\cite{Salamon}. The main point is the compatibility between hyperkähler quotients and twistor spaces from Hitchin-Karlhede-Lindström-Roček \cite{hklr} p.$560$, see also Hitchin \cite{Hitchin_1992_hk}. The following Theorem as well as its proof was told  to us by Nakajima.
\begin{theorem}\label{th_cbvdb}
For $\theta$ generic in $\mathbb{H}_0\otimes_{\RR} Z(\mathfrak{u}_v)$ define
\[
\mathbb{H}_0^*.\theta = \left\lbrace h\otimes \theta  \left| h\in \mathbb{H}_0^*\right. \right\rbrace.
\]
There exists a diffeomorphism $f$ such that the following diagram commutes
\[
\begin{tikzcd}
\mu_{\mathbb{H}}\inv(\mathbb{H}_0^*.\theta)/U_v \arrow[r,"f"] \arrow[rd,"\mu_{\mathbb{H}}",swap]& \mu_{\mathbb{H}}\inv(\theta)/U_v \times \mathbb{H}_0^*.\theta \arrow[d]\\
 & \mathbb{H}_0^*.\theta 
\end{tikzcd}
\]
the vertical arrow is the projection to $\mathbb{H}_0^*.\theta$.
\end{theorem}
\begin{proof}
Consider the quaternionic vector space $\Rep\left(\widetilde{\Gamma},v\right)$ and the projection
\[
\Rep\left(\widetilde{\Gamma},v\right)\times \mathcal{S}^2 \to \mathcal{S}^2.
\]
With $\mathcal{S}^2$ the $2$-sphere of imaginary quaternions with unit norm
\[
\mathcal{S}^2=\left\lbrace a I + b J + c K \left| a^2+b^2+c^2 = 1 \right. \right\rbrace.
\]
$\mathcal{S}^2$ is given the usual complex structure of the projective line. The twistor space associated to $\Rep\left(\widetilde{\Gamma},v\right)$ is the manifold $\Rep\left(\widetilde{\Gamma},v\right)\times \mathcal{S}^2$ endowed with a complex structure such that the fiber over $I_u\in \mathcal{S}^2$ is $\Rep\left(\widetilde{\Gamma},v\right)$ seen as a vector space with complex structure $I_u$.

As detailed in \cite{CB_VDB}, the group of quaternion of unit norm, identified with $\SU(2)$, acts on $\mathbb{H}^0\otimes Z(\uv)$ by \[
h.\left( h'\otimes \theta \right)= h h' \overline{h} \otimes \theta.
\]
with $\overline{ a I + b J +c K + d }={ -a I - b J -c K + d }$. Let $\theta$ a generic parameter, up to the choice of orthonormal basis of $\mathbb{R}^3$ we can assume $\theta= I \otimes \theta_I$. The $\SU(2)$ orbit of $\theta$ thus identifies with $\mathcal{S}^2$ as
\begin{equation}\label{eq_su2}
\SU(2).\theta = \left\lbrace I_u \otimes \theta \left| I_u \in \mathcal{S}^2\right. \right\rbrace.
\end{equation}
The twistor space of the hyperkähler manifold $\mu_{\mathbb{H}}\inv(\theta) /U_v$ is a complex manifold $\mathcal{T}$ with an holomorphic map $p$ to $\mathcal{S}^2$
\[
\begin{tikzcd}
\mathcal{T}\arrow[r,"p"] & \mathcal{S}^2.
\end{tikzcd}
\]
The underlying differential manifold of the twistor space is just a product and $p$ the projection to the second factor
\[
\begin{tikzcd}
\mu_{\mathbb{H}}\inv(\theta) /U_v\times \mathcal{S}^2 \arrow[r]&\mathcal{S}^2.
\end{tikzcd}
\]
The twistor spaces construction is compatible with hyperkähler quotients as explained in \cite{hklr} p.$560$. Thus the fiber of $p$ over $I_u$ is $\mu_{\mathbb{H}}\inv(\theta)/U_v$ endowed with the complex structure inherited from the complex structure $I_u$ on $\Rep\left(\overline{\Gamma},v\right)$. Namely if $I_u \otimes \theta = \left(\theta_I',\theta_J',\theta_K'\right)$ then the fiber of the twistor space over $I_u$ is the complex manifold
\[
p\inv(I_u) = \mu_{\mathbb{C}}\inv(\theta_J'+i\theta_K') \cap\mu_{\mathbb{R}}\inv(\theta_I')/U_v
\]
Thus fibers of $p$ are exactly fibers of $\mu_{\mathbb{H}}$ and the twistor space provides trivialization of the hyperkähler moment map over the orbit $\SU(2).\theta$:
\[
\begin{tikzcd}[column sep=tiny]
\mu_{\mathbb{H}}\inv (\SU(2).\theta)/ U_v \arrow[d,"\mu_{\mathbb{H}}",swap] \arrow[rr,"\beta"] & & \mathcal{T}   \arrow[dr,"p",swap] \arrow[rr,"\gamma"] & & \mu_{\mathbb{H}}\inv(\theta)/U_v\times \mathcal{S}^2\arrow[dl] \\
\arrow[rrr,"\sim","\alpha"']  \SU(2).\theta & && \mathcal{S}^2 &
\end{tikzcd}
\]
$\alpha$ is defined thanks to \eqref{eq_su2}, the map $\beta$ is the identity on the fibers and $\gamma$ forgets the complex structure. This diagram traduces the equivalence between, on the right, varying complex structure on a fixed fiber $\mu_{\mathbb{H}}\inv(\theta)/U_v$ and on the left varying the fiber for a fixed complex structure $I$. 

The construction is similar to Crawley-Boevey and Van den Bergh's construction except that the twistor space formalism allows to obtain a trivialization over the non-contractible space $\SU(2).\theta$.

As in \cite{CB_VDB}, the trivialization can be extended thanks to the $\RR_{>0}$ action. Note that for $t$ a positive real number $\mu_{\mathbb{H}}(t x)= t^2 \mu_{\mathbb{H}}(x)$. Then identifying $\mathcal{S}^2\times \mathbb{R}_{>0}$ with $\mathbb{H}_0^*$ we obtain the trivialization
\[
\begin{tikzcd}
\mu_{\mathbb{H}}\inv (\mathbb{H}_0^* .\theta)/ U_v \arrow[d] \arrow[r]  & \mu_{\mathbb{H}}\inv (\theta) / U_v \times \mathbb{H}_0^* \arrow[d] \\
 \mathbb{H}_0^* .\theta \arrow[r]  & \mathbb{H}_0^*
\end{tikzcd}
\]
The $\SU(2)$-action on the base of this trivialization traduces the variation of complex structure on the hyperkähler manifold $\mu_{\mathbb{H}}\inv(\theta)/U_v$ whereas the $\RR_{>0}$ action traduces the rescaling of the metric.
\end{proof}

\bibliographystyle{alpha}
\bibliography{biblio2} 
\end{document}